\documentclass[10pt,draft]{article}
\usepackage{graphicx}
\usepackage{amssymb}
\usepackage{graphicx}
\usepackage{amsmath,amsthm,amscd}
\usepackage[latin1]{inputenc}
\numberwithin{equation}{section}
\newtheorem{theorem}{Theorem}

\newtheorem{proposition}[theorem]{Proposition}
\newtheorem{lemma}[theorem]{Lemma}

\newtheorem{corollary}[theorem]{Corollary}

\newtheorem{definition}[theorem]{Definition}
\theoremstyle{definition}
\newtheorem{example}[theorem]{Example}

\newtheorem{remark}[theorem]{Remark}

\title{\textbf{Gröbner bases of modules\\over $\sigma-PBW$ extensions}}
\author{Haydee Jiménez \footnote{Graduate student, Universidad de Sevilla, Spain.}
\ \& Oswaldo Lezama\\
Seminario de Álgebra Constructiva - $\text{SAC}^2$\\
Departamento de Matemáticas\\
Universidad Nacional de Colombia, Bogotá, COLOMBIA\\
}
\date{}
\begin{document}
\maketitle \setlength{\parindent}{0pt}
\begin{abstract}\noindent For $\sigma-PWB$ extensions, we extend to modules the theory of Gröbner bases of left
ideals presented in \cite{Gallego2}. As an application, if $A$ is
a bijective quasi-commutative $\sigma-PWB$ extension, we compute
the module of syzygies of a submodule of the free module $A^m$.

\bigskip

\bigskip


\bigskip

\bigskip

\end{abstract}

\newpage

\section{Introduction}

In this paper we present the theory of Gröbner bases for submodules of $A^m$, $m\geq 1$, where
$A=\sigma(R)\langle x_1,\dots,x_n\rangle$ is a $\sigma-PBW$ extension of $R$, with $R$ a $LGS$ ring
(see Definition \ref{LGSring}) and $Mon(A)$ endowed with some monomial order (see Definition
\ref{monomialorder}). $A^m$ is the left free $A$-module of column vectors of length $m\geq 1$; if
$A$ is bijective, $A$ is a left Noetherian ring (see \cite{Lezama3}), then $A$ is an $IBN$ ring
(Invariant Basis Number), and hence, all bases of the free module $A^m$ have $m$ elements. Note
moreover that $A^m$ is a left Noetherian, and hence, any submodule of $A^m$ is finitely generated.
The main purpose is to define and calculate Gröbner bases for submodules of $A^m$, thus, we will
define the monomials in $A^m$, orders on the monomials, the concept of reduction, we will construct
a Division Algorithm, we will give equivalent conditions in order to define Gröbner bases, and
finally, we will compute Gröbner bases using a procedure similar to Buchberger's Algorithm in the
particular case of quasi-commutative bijective $\sigma-PBW$ extensions. The results presented here
generalize those of \cite{Gallego2} where $\sigma$-PBW extensions were defined and the theory of
Gröbner bases for the left ideals was constructed. Most of proofs are easily adapted from
\cite{Gallego2} and hence we will omit them. As an application, the final section of the paper
concerns with the computation of the module of syzygies of a given submodule of $A^m$ for the
particular case when $A$ is bijective quasi-commutative.
\begin{definition}\label{gpbwextension}
Let $R$ and $A$ be rings, we say that $A$ is a $\sigma-PBW$
extension of $R$ {\rm(}or skew $PBW$ extension{\rm)}, if the
following conditions hold:
\begin{enumerate}
\item[\rm (i)]$R\subseteq A$.
\item[\rm (ii)]There exist finite elements $x_1,\dots ,x_n\in A-R$ such $A$ is a left $R$-free module with basis
\begin{center}
$Mon(A):= \{x^{\alpha}=x_1^{\alpha_1}\cdots
x_n^{\alpha_n}|\alpha=(\alpha_1,\dots ,\alpha_n)\in
\mathbb{N}^n\}$.
\end{center}
In this case we say also that $A$ is a left polynomial ring over
$R$ with respect to $\{x_1,\dots,x_n\}$ and $Mon(A)$ is the set of
standard monomials of $A$. Moreover, $x_1^0\cdots x_n^0:=1\in
Mon(A)$.
\item[\rm (iii)]For every $1\leq i\leq n$ and $r\in R-\{0\}$ there exists $c_{i,r}\in R-\{0\}$ such that
\begin{equation}\label{sigmadefinicion1}
x_ir-c_{i,r}x_i\in R.
\end{equation}
\item[\rm (iv)]For every $1\leq i,j\leq n$ there exists $c_{i,j}\in R-\{0\}$ such that
\begin{equation}\label{sigmadefinicion2}
x_jx_i-c_{i,j}x_ix_j\in R+Rx_1+\cdots +Rx_n.
\end{equation}
Under these conditions we will write $A=\sigma(R)\langle x_1,\dots
,x_n\rangle$.
\end{enumerate}
\end{definition}
The following proposition justifies the notation that we have
introduced for the skew $PBW$ extensions.
\begin{proposition}\label{sigmadefinition}
Let $A$ be a $\sigma-PBW$ extension of $R$. Then, for every $1\leq
i\leq n$, there exist an injective ring endomorphism
$\sigma_i:R\rightarrow R$ and a $\sigma_i$-derivation
$\delta_i:R\rightarrow R$ such that
\begin{center}
$x_ir=\sigma_i(r)x_i+\delta_i(r)$,
\end{center}
for each $r\in R$.
\end{proposition}
\begin{proof}
See \cite{Gallego2}.
\end{proof}
A particular case of $\sigma-PBW$ extension is when all
derivations $\delta_i$ are zero. Another interesting case is when
all $\sigma_i$ are bijective. We have the following definition.
\begin{definition}\label{sigmapbwderivationtype}
Let $A$ be a $\sigma-PBW$ extension.
\begin{enumerate}
\item[\rm (a)]
$A$ is quasi-commutative if the conditions (iii) and (iv) in the
Definition \ref{gpbwextension} are replaced by
\begin{enumerate}
\item[\rm ($iii'$)]For every $1\leq i\leq n$ and $r\in R-\{0\}$ there exists $c_{i,r}\in R-\{0\}$ such that
\begin{equation}
x_ir=c_{i,r}x_i.
\end{equation}
\item[\rm ($iv'$)]For every $1\leq i,j\leq n$ there exists $c_{i,j}\in R-\{0\}$ such that
\begin{equation}
x_jx_i=c_{i,j}x_ix_j.
\end{equation}
\end{enumerate}
\item[\rm (b)]$A$ is bijective if $\sigma_i$ is bijective for
every $1\leq i\leq n$ and $c_{i,j}$ is invertible for any $1\leq
i<j\leq n$.
\end{enumerate}
\end{definition}
Some interesting examples of $\sigma-PBW$ extensions were given in
\cite{Gallego2}. We repeat next some of them without details.
\begin{example}\label{gpbwexample}
(i) Any $PBW$ extension (see \cite{Bell}) is a bijective
$\sigma-PBW$ extension.

(ii) Any skew polynomial ring $R[x;\sigma ,\delta]$, with $\sigma$
injective, is a $\sigma-PBW$ extension; in this case we have
$R[x;\sigma ,\delta]=\sigma(R)\langle x\rangle$. If additionally
$\delta=0$, then $R[x;\sigma]$ is quasi-commutative.

(iii) Any iterated skew polynomial ring $R[x_1;\sigma_1
,\delta_1]\cdots [x_n;\sigma_n ,\delta_n]$ is a $\sigma-PBW$
extension if it satisfies the following conditions:
\begin{center}
For $1\leq i\leq n$, $\sigma_i$ is injective.

For every $r\in R$ and $1\leq i\leq n$,
$\sigma_i(r),\delta_i(r)\in R$.

For $i<j$, $\sigma_j(x_i)=cx_i+d$, with $c,d\in R$,\ and $c$ has a
left inverse.

For $i<j$, $\delta_j(x_i)\in R+Rx_1+\cdots +Rx_i$.
\end{center}
Under these conditions we have
\begin{center}
$R[x_1;\sigma_1 ,\delta_1]\cdots [x_n;\sigma_n
,\delta_n]=\sigma(R)\langle x_1,\dots,x_n\rangle$.
\end{center}
In particular, any Ore algebra $K[t_1,\dots,t_m][x_1;\sigma_1
,\delta_1]\cdots [x_n;\sigma_n ,\delta_n]$ ($K$ a field) is a
$\sigma-PBW$ extension if it satisfies the following condition:
\begin{center}
For $1\leq i\leq n$, $\sigma_i$ is injective.
\end{center}
Some concrete examples of Ore algebras of injective type are the
following.

The algebra of shift operators: let $h\in K$, then the algebra of
shift operators is defined by $S_h:=K[t][x_h;\sigma_h,\delta_h]$,
where $\sigma_h(p(t)):=p(t-h)$, and $\delta_h:=0$ (observe that
$S_h$ can be considered also as a skew polynomial ring of
injective type). Thus, $S_h$ is a quasi-commutative bijective
$\sigma-PBW$ extension.

The mixed algebra $D_h$: let again $h\in K$, then the mixed
algebra $D_h$ is defined by
$D_h:=K[t][x;i_{K[t]},\frac{d}{dt}][x_h;\sigma_h,\delta_h]$, where
$\sigma_h(x):=x$. Then, $D_h$ is a quasi-commutative bijective
$\sigma-PBW$ extension.

The algebra for multidimensional discrete linear systems is
defined by
$D:=K[t_1,\dots,t_n][x_1,\sigma_1,0]\cdots[x_n;\sigma_n,0]$, where
\begin{center}
$\sigma_i(p(t_1,\dots,t_n)):=p(t_1,\dots,t_{i-1},t_{i}+1,t_{i+1},\dots,t_n),
\ \sigma_i(x_i)=x_i$, $1\leq i\leq n$.
\end{center}
$D$ is a quasi-commutative bijective $\sigma-PBW$ extension.

(iv) Additive analogue of the Weyl algebra: let $K$ be a field,
the $K$-algebra $A_n(q_1,\dots,q_n)$ is generated by
$x_1,\dots,x_n,y_1,\dots,y_n$ and subject to the relations:
\begin{center}
$x_jx_i = x_ix_j, y_jy_i = y_iy_j, \ 1 \leq i,j \leq n$,

$y_ix_j=x_jy_i, \ i\neq j$,

$y_ix_i = q_ix_iy_i + 1, \ 1\leq i\leq n$,
\end{center}
where $q_i\in K-\{0\}$. $A_n(q_1, \dots, q_n)$ satisfies the
conditions of (iii) and is bijective; we have
\begin{center}
$A_n(q_1, \dots, q_n)=\sigma(K[x_1,\dots,x_n])\langle
y_1,\dots,y_n\rangle$.
\end{center}

(v) Multiplicative analogue of the Weyl algebra: let $K$ be a
field, the $K$-algebra $\mathcal{O}_n(\lambda_{ji})$ is generated
by $x_1,\dots,x_n$ and subject to the relations:
\begin{center}
$x_jx_i =\lambda_{ji}x_ix_j ,\ 1\leq i<j\leq n$,
\end{center}
where $\lambda_{ji}\in K-\{0\}$. $\mathcal{O}_n(\lambda_{ji})$
satisfies the conditions of (iii), and hence
\begin{center}
$\mathcal{O}_n(\lambda_{ji})=\sigma(K[x_1])\langle
x_2,\dots,x_n\rangle$.
\end{center}
Note that $\mathcal{O}_n(\lambda_{ji})$ is quasi-commutative and
bijective.

(vi) $q$-Heisenberg algebra: let $K$ be a field , the $K$-algebra
$h_n(q)$ is generated by
$x_1,\dots,x_n,y_1,\dots,y_n,z_1,\dots,z_n$ and subject to the
relations:
\begin{center}
$x_jx_i = x_ix_j, z_jz_i=z_iz_j, y_jy_i = y_iy_j, \ 1 \leq i,j
\leq n$,

$z_jy_i=y_iz_j,z_jx_i=x_iz_j,y_jx_i=x_iy_j, \ i\neq j$,

$z_iy_i = qy_iz_i,z_ix_i=q^{-1}x_iz_i+y_i,y_ix_i = qx_iy_i, \
1\leq i\leq n$,
\end{center}
with $q\in K-\{0\}$. $h_n(q)$ is a bijective $\sigma-PBW$
extension of $K$:
\begin{center}
$h_n(q)=\sigma(K)\langle
x_1,\dots,x_n;y_1,\dots,y_n;z_1,\dots,z_n\rangle$.
\end{center}
(vi) Many other examples are presented in \cite{Lezama3}.
\end{example}
\begin{definition}
Let $A$ be a $\sigma-PBW$ extension of $R$ with endomorphisms
$\sigma_i$, $1\leq i\leq n$, as in Proposition
\ref{sigmadefinition}.
\begin{enumerate}
\item[\rm (i)]For $\alpha=(\alpha_1,\dots,\alpha_n)\in \mathbb{N}^n$,
$\sigma^{\alpha}:=\sigma_1^{\alpha_1}\cdots \sigma_n^{\alpha_n}$,
$|\alpha|:=\alpha_1+\cdots+\alpha_n$. If
$\beta=(\beta_1,\dots,\beta_n)\in \mathbb{N}^n$, then
$\alpha+\beta:=(\alpha_1+\beta_1,\dots,\alpha_n+\beta_n)$.
\item[\rm (ii)]For $X=x^{\alpha}\in Mon(A)$,
$\exp(X):=\alpha$ and $\deg(X):=|\alpha|$.
\item[\rm (iii)]Let $0\neq f\in A$, $t(f)$ is the finite
set of terms that conform $f$, i.e., if $f=c_1X_1+\cdots +c_tX_t$,
with $X_i\in Mon(A)$ and $c_i\in R-\{0\}$, then
$t(f):=\{c_1X_1,\dots,c_tX_t\}$.
\item[\rm (iv)]Let $f$ be as in {\rm(iii)}, then $\deg(f):=\max\{\deg(X_i)\}_{i=1}^t.$
\end{enumerate}
\end{definition}
The $\sigma-PBW$ extensions can be characterized in a similar way
as was done in \cite{Gomez-Torrecillas} for $PBW$ rings.
\begin{theorem}\label{coefficientes}
Let $A$ be a left polynomial ring over $R$ w.r.t
$\{x_1,\dots,x_n\}$. $A$ is a $\sigma-PBW$ extension of $R$ if and
only if the following conditions hold:
\begin{enumerate}
\item[\rm (a)]For every $x^{\alpha}\in Mon(A)$ and every $0\neq
r\in R$ there exists unique elements
$r_{\alpha}:=\sigma^{\alpha}(r)\in R-\{0\}$ and $p_{\alpha ,r}\in
A$ such that
\begin{equation}\label{611}
x^{\alpha}r=r_{\alpha}x^{\alpha}+p_{\alpha , r},
\end{equation}
where $p_{\alpha ,r}=0$ or $\deg(p_{\alpha ,r})<|\alpha|$ if
$p_{\alpha , r}\neq 0$. Moreover, if $r$ is left invertible, then
$r_\alpha$ is left invertible.

\item[\rm (b)]For every $x^{\alpha},x^{\beta}\in Mon(A)$ there
exist unique elements $c_{\alpha,\beta}\in R$ and
$p_{\alpha,\beta}\in A$ such that
\begin{equation}\label{612}
x^{\alpha}x^{\beta}=c_{\alpha,\beta}x^{\alpha+\beta}+p_{\alpha,\beta},
\end{equation}
where $c_{\alpha,\beta}$ is left invertible, $p_{\alpha,\beta}=0$
or $\deg(p_{\alpha,\beta})<|\alpha+\beta|$ if
$p_{\alpha,\beta}\neq 0$.
\end{enumerate}
\end{theorem}
\begin{proof}
See \cite{Gallego2}.
\end{proof}
\begin{remark}\label{identities}
(i) A left inverse of $c_{\alpha,\beta}$ will be denoted by
$c_{\alpha,\beta}'$. We observe that if $\alpha=0$ or $\beta=0$,
then $c_{\alpha,\beta}=1$ and hence $c_{\alpha,\beta}'=1$.

(ii) Let $\theta,\gamma,\beta\in \mathbb{N}^n$ and $c\in R$, then
we it is easy to check the following identities:
\begin{center}
$\sigma^\theta(c_{\gamma,\beta})c_{\theta,\gamma+\beta}=c_{\theta,\gamma}c_{\theta+\gamma,\beta}$,

$\sigma^\theta(\sigma^\gamma
(c))c_{\theta,\gamma}=c_{\theta,\gamma}\sigma^{\theta+\gamma}(c)$.
\end{center}

(iii) We observe if $A$ is a $\sigma-PBW$ extension
quasi-commutative, then from the proof of Theorem
\ref{coefficientes} (see \cite{Gallego2}) we conclude that
$p_{\alpha,r}=0$ and $p_{\alpha,\beta}=0$, for every $0\neq r\in
R$ and every $\alpha,\beta \in \mathbb{N}^n$.

(iv) We have also that if $A$ is a bijective $\sigma-PBW$
extension, then $c_{\alpha,\beta}$ is invertible for any
$\alpha,\beta\in \mathbb{N}^n$.
\end{remark}
A key property of $\sigma$-PBW extensions is the content of the
following theorem.
\begin{theorem}\label{hilbertbasisforgpbw}
Let $A$ be a bijective skew $PBW$ extension of $R$. If $R$ is a
left Noetherian ring then $A$ is also a left Noetherian ring.
\end{theorem}
\begin{proof}
See \cite{Lezama3}.
\end{proof}
Let $A=\sigma(R)\langle x_1, \dots , x_n\rangle$ be a $\sigma-PBW$
extension of $R$ and let $\succeq$ be a total order defined on
$Mon(A)$. If $x^{\alpha}\succeq x^{\beta}$ but $x^{\alpha}\neq
x^{\beta}$ we will write $x^{\alpha}\succ x^{\beta}$. Let $f\neq
0$ be a polynomial of $A$, if
\begin{center}
$f=c_1X_1+\cdots +c_tX_t$,
\end{center}
with $c_i\in R-\{0\}$ and $X_1\succ \cdots \succ X_t$ are the
monomials of $f$, then $lm(f):=X_1$ is the \textit{leading
monomial} of $f$, $lc(f):=c_1$ is the \textit{leading coefficient}
of $f$ and $lt(f):=c_1X_1$ is the \textit{leading term} of $f$. If
$f=0$, we define $lm(0):=0,lc(0):=0,lt(0):=0$, and we set $X\succ
0$ for any $X\in Mon(A)$. Thus, we extend $\succeq$ to $Mon(A)\cup
\{0\}$.

\begin{definition}\label{monomialorder}
Let $\succeq$ be a total order on $Mon(A)$, we say that $\succeq$
is a monomial order on $Mon(A)$ if the following conditions hold:
\begin{enumerate}
\item[\rm (i)]For every $x^{\beta},x^{\alpha},x^{\gamma},x^{\lambda}\in Mon(A)$
\begin{center}
$x^{\beta}\succeq x^{\alpha}$ $\Rightarrow$
$lm(x^{\gamma}x^{\beta}x^{\lambda})\succeq
lm(x^{\gamma}x^{\alpha}x^{\lambda})$.
\end{center}

\item[\rm (ii)]$x^{\alpha}\succeq 1$, for every $x^{\alpha}\in
Mon(A)$.
\item[\rm (iii)]$\succeq$ is degree compatible, i.e., $|\beta|\geq |\alpha|\Rightarrow x^{\beta}\succeq
x^{\alpha}$.
\end{enumerate}
\end{definition}
Monomial orders are also called \textit{admissible orders}. From
now on we will assume that $Mon(A)$ is endowed with some monomial
order.
\begin{definition}
Let $x^{\alpha},x^{\beta}\in Mon(A)$, we say that $x^{\alpha}$
divides $x^{\beta}$, denoted by $x^{\alpha}|x^{\beta}$, if there
exists $x^{\gamma},x^{\lambda}\in Mon(A)$ such that
$x^{\beta}=lm(x^{\gamma}x^{\alpha}x^{\lambda})$.
\end{definition}
\begin{proposition}
Let $x^{\alpha},x^{\beta}\in Mon(A)$ and $f,g\in A-\{0\}$. Then,
\begin{enumerate}
\item[\rm (a)]
$lm(x^{\alpha}g)=lm(x^{\alpha}lm(g))=x^{\alpha+\exp(lm(g))}$. In
particular,
\begin{center}
$lm(lm(f)lm(g))=x^{\exp(lm(f))+\exp(lm(g))}$
\end{center}
and
\begin{equation}
lm(x^{\alpha}x^{\beta})=x^{\alpha+\beta}.\label{divrelation}
\end{equation}
\item[\rm (b)]The
following conditions are equivalent:
\begin{enumerate}
\item[\rm (i)]$x^{\alpha}|x^{\beta}$.
\item[\rm (ii)]There exists a unique $x^{\theta}\in Mon(A)$ such that
$x^{\beta}=lm(x^{\theta}x^{\alpha})=x^{\theta+\alpha}$ and hence
$\beta=\theta+\alpha$.
\item[\rm (iii)]There exists a unique $x^{\theta}\in Mon(A)$ such that
$x^{\beta}=lm(x^{\alpha}x^{\theta})=x^{\alpha+\theta}$ and hence
$\beta=\alpha+\theta$.
\item[\rm (iv)]$\beta_i\geq \alpha_i$ for $1\leq i\leq n$, with
$\beta:=(\beta_1,\dots,\beta_n)$ and
$\alpha:=(\alpha_1,\dots,\alpha_n)$.
\end{enumerate}
\end{enumerate}
\end{proposition}
\begin{proof}
See \cite{Gallego2}.
\end{proof}
We note that a least common multiple of monomials of $Mon(A)$
there exists: in fact, let $x^{\alpha},x^{\beta}\in Mon(A)$, then
$lcm(x^{\alpha},x^{\beta})=x^{\gamma}\in Mon(A)$, where
$\gamma=(\gamma_1,\dots,\gamma_n)$ with
$\gamma_i:=\max\{\alpha_i,\beta_i\}$ for each $1\leq i\leq n$.

\bigskip

Some natural computational conditions on $R$ will be assumed in
the rest of this paper (compare with \cite{Lezama2}).
\begin{definition}\label{LGSring}
A ring $R$ is left Gröbner soluble {\rm(}$LGS${\rm)} if the
following conditions hold:
\begin{enumerate}
\item[\rm (i)]$R$ is left Noetherian.
\item[\rm (ii)]Given $a,r_1,\dots,r_m\in R$ there exists an
algorithm which decides whether $a$ is in the left ideal
$Rr_1+\cdots+Rr_m$, and if so, find $b_1,\dots,b_m\in R$ such that
$a=b_1r_1+\cdots+b_mr_m$.
\item[\rm (iii)]Given $r_1,\dots,r_m\in R$ there exists an
algorithm which finds a finite set of generators of the left
$R$-module
\begin{center}
$Syz_R[r_1\ \cdots \ r_m]:=\{(b_1,\dots,b_m)\in
R^m|b_1r_1+\cdots+b_mr_m=0\}$.
\end{center}
\end{enumerate}
\end{definition}
The three above conditions imposed to $R$ are needed in order to
guarantee a Gröbner theory in the rings of coefficients, in
particular, to have an effective solution of the membership
problem in $R$ (see (ii) in Definition \ref{reductionformodules}
below). From now on we will assume that $A=\sigma(R)\langle
x_1,\dots,x_n\rangle$ is a $\sigma-PBW$ extension of $R$, where
$R$ is a $LGS$ ring and $Mon(A)$ is endowed with some monomial
order.

\bigskip

We conclude this chapter with a remark about some other classes of
noncommutative rings of polynomial type close related with
$\sigma$-PBW extensions.
\begin{remark}
(i) Viktor Levandovskyy has defined in \cite{Levandovskyy} the
$G$-algebras and has constructed the theory of Gröbner bases for
them. Let $K$ be a field, a $K$-algebra $A$ is called a
\textit{$G$-algebra} if $K\subset Z(A)$ (center of $A$) and $A$ is
generated by a finite set $\{x_1, \ldots, x_n\}$ of elements that
satisfy the following conditions: (a) the collection of standard
monomials of $A$, $Mon(A)=Mon(\{x_1, \ldots, x_n\})$, is a
$K$-basis of $A$. (b) $x_jx_i= c_{ij}x_ix_j+d_{ij}$, for $1 \leq i
< j \leq n$, with $c_{ij} \in K^*$ and $d_{ij}\in A$. (c) There
exists a total order $<_A$ on $Mon(A)$ such that for $i<j$,
$lm(d_{ij}) <_A x_ix_j$. (d) For $1 \leq i < j < k \leq n$,
$c_{ik}c_{jk}d_{ij}x_k - x_kd_{ij} + c_{jk}x_jd_{ik} -
c_{ij}d_{ik}x_j + d_{jk}x_i - c_{ij}c_{ik}x_id_{jk}=0$. According
to this definition, the coefficients of a polynomial in a
$G$-algebra are in a field and they commute with the variables
$x_1,\dots,x_n$. From this, and also from (c) and (d), we conclude
that the class of $G$-algebras does not coincide with the class of
$\sigma$-PBW extensions. However, the intersection of these two
classes of rings is not empty. In fact, the universal enveloping
algebra of a finite dimensional Lie algebra, Weyl algebras and the
additive or multiplicative analogue of a Weyl algebra, are
$G$-algebras and also $\sigma$-PBW extensions.

(ii) A similar remark can be done with respect to $PBW$ rings and
algebras defined by Bueso, Gómez-Torrecillas and Verschoren in
\cite{Gomez-Torrecillas2}.
\end{remark}

\section{Monomial orders on $Mon(A^m)$}
We will often write the elements of $A^m$ also as row vectors if
this not represent confusion. We recall that the canonical basis
of $A^m$ is
\begin{center}
$\textbf{\emph{e}}_1=(1,0,\dots
,0),\textbf{\emph{e}}_2=(0,1,0,\dots, 0),\dots
,\textbf{\emph{e}}_m=(0,0,\dots ,1)$.
\end{center}
\begin{definition}
A monomial in $A^m$ is a vector $\textbf{X}=X\textbf{e}_i$, where
$X=x^{\alpha}\in Mon(A)$ and $1\leq i\leq m$, i.e.,
\begin{center}
$\textbf{X}=X\textbf{e}_i=(0,\dots ,X,\dots ,0)$,
\end{center}
where $X$ is in the $i$-th position, named the index of
$\textbf{X}$, $ind(\textbf{X}):=i$. A term is a vector
$c\textbf{X}$, where $c\in R$. The set of monomials of $A^m$ will
be denoted by $\textrm{Mon}(A^m)$. Let
$\textbf{Y}=Y\textbf{e}_j\in Mon(A^m)$, we say that $\textbf{X}$
divides $\textbf{Y}$ if $i=j$ and $X$ divides $Y$. We will say
that any monomial $\textbf{X}\in Mon(A^m)$ divides the null vector
$\textbf{\emph{0}}$. The least common multiple of $\textbf{X}$ and
$\textbf{Y}$, denoted by $lcm(\textbf{X},\textbf{Y})$, is
$\textbf{\emph{0}}$ if $i\neq j$, and $U\textbf{e}_i$, where
$U=lcm(X,Y)$, if $i=j$. Finally, we define
$\exp(\textbf{X}):=\exp(X)=\alpha$ and
$\deg(\textbf{X}):=\deg(X)=|\alpha|$.
\end{definition}
We now define monomials orders on $Mon(A^m)$.
\begin{definition}
A monomial order on $Mon(A^m)$ is a total order $\succeq$
satisfying the following three conditions:
\begin{enumerate}
\item[\rm(i)] $lm(x^{\beta}x^{\alpha})\textbf{e}_{i}\succeq x^{\alpha}\textbf{e}_{i}$, for every
monomial $\textbf{X}=x^{\alpha}\textbf{e}_{i}\in Mon(A^{m})$ and
any monomial $x^{\beta}$ in $Mon(A)$.
\item[\rm(ii)] If $\textbf{Y}=x^{\beta}\textbf{e}_{j}\succeq \textbf{X}=x^{\alpha}\textbf{e}_{i}$, then
$lm(x^{\gamma}x^{\beta})\textbf{e}_{j}\succeq
lm(x^{\gamma}x^{\alpha})\textbf{e}_{i}$ for all
$\textbf{X},\textbf{Y}\in Mon(A^m)$ and every $x^{\gamma}\in
Mon(A)$.
\item[\rm (iii)]$\succeq$ is degree compatible, i.e., $\deg(\textbf{X})\geq \deg(\textbf{Y})\Rightarrow \textbf{X}\succeq
\textbf{Y}$.
\end{enumerate}
If $\textbf{X}\succeq \textbf{Y}$ but $\textbf{X}\neq \textbf{Y}$
we will write $\textbf{X}\succ \textbf{Y}$.
$\textbf{Y}\preceq\textbf{X}$ means that
$\textbf{X}\succeq\textbf{Y}$.
\end{definition}
\begin{proposition}\label{wellorder}
Every monomial order on $Mon(A^m)$ is a well order.
\end{proposition}
\begin{proof}
We can easy adapt the proof for left ideals presented in
\cite{Gallego2}.
\end{proof}
Given a monomial order $\succeq$ on
$Mon(A)$, we can define two natural orders on $Mon(A^m)$.
\begin{definition}
Let $\textbf{X}=X\textbf{e}_i$ and $\textbf{Y}=Y\textbf{e}_j\in
Mon(A^m)$.
\begin{enumerate}
\item [\rm(i)]The TOP {\rm(}term over position{\rm)} order is defined by
\begin{center}
$\textbf{X}\succeq\textbf{Y}\Longleftrightarrow
\begin{cases} X\succeq Y & \\
\text{or} & \\
X=Y \text{and}& i>j.
\end{cases}$
\end{center}
\item[\rm(ii)]The TOPREV order is defined by
\begin{center}
$\textbf{X}\succeq\textbf{Y}\Longleftrightarrow
\begin{cases} X\succeq Y & \\
\text{or} & \\
X=Y \text{and}& i<j.
\end{cases}$
\end{center}
\end{enumerate}
\end{definition}
\begin{remark}
(i) Note that with TOP we have
\begin{center}
$\textbf{\emph{e}}_m\succ\textbf{\emph{e}}_{m-1}\succ\cdots
\succ\textbf{\emph{e}}_1$
\end{center}
and
\begin{center}
$\textbf{\emph{e}}_1\succ\textbf{\emph{e}}_{2}\succ\cdots
\succ\textbf{\emph{e}}_m$
\end{center}
for TOPREV.

(ii) The POT (position over term) and POTREV  orders defined in
\cite{Loustaunau} and \cite{Lezama2} for modules over classical
polynomial commutative rings are not degree compatible.

(iii) Other examples of monomial orders in $Mon(A^m)$ are
considered in \cite{Gomez-Torrecillas2}.
\end{remark}
We fix monomial orders on $Mon(A)$ and $Mon(A^m)$; let
$\textbf{\emph{f}}\neq \textbf{0}$ be a vector of $A^m$, then we
may write $\textbf{\emph{f}}$ as a sum of terms in the following
way
\begin{center}
$\textbf{\emph{f}}=c_1\textbf{\emph{X}}_1+\cdots
+c_t\textbf{\emph{X}}_t$,
\end{center}
where $c_1,\dots ,c_t\in R-\{0\}$ and
$\textbf{\emph{X}}_1\succ\textbf{\emph{X}}_2\succ\cdots
\succ\textbf{\emph{X}}_t$ are monomials of $Mon(A^m)$.
\begin{definition}
With the above notation, we say that
\begin{enumerate}
\item[\rm(i)]$lt(\textbf{f}):=c_1\textbf{X}_1$ is the leading term of $\textbf{f}$.
\item[\rm(ii)]$lc(\textbf{f}):=c_1$ is the leading coefficient of $\textbf{f}$.
\item[\rm(iii)]$lm(\textbf{f}):=\textbf{X}_1$ is the leading monomial of $\textbf{f}$.
\end{enumerate}
\end{definition}
For $\textbf{\emph{f}}=\textbf{0}$ we define
$lm(\textbf{0})=\textbf{0},lc(\textbf{0})=0,lt(\textbf{0})=\textbf{0}$,
and if $\succeq$ is a monomial order on $Mon(A^m)$, then we define
$\textbf{X}\succ\textbf{0}$ for any $\textbf{X}\in Mon(A^m)$. So,
we extend $\succeq$ to $Mon(A^m)\cup \{\textbf{0}\}$.

\section{Reduction in $A^m$}

The reduction process in $A^m$ is defined as follows.
\begin{definition}\label{reductionformodules}
Let $F$ be a finite set of non-zero vectors of $A^m$, and let
$\textbf{f},\textbf{h}\in A^m$, we say that $\textbf{f}$ reduces
to $\textbf{h}$ by $F$ in one step, denoted
$\textbf{f}\xrightarrow{\,\, F\,\, } \textbf{h}$, if there exist
elements $\textbf{f}_1,\dots,\textbf{f}_t\in F$ and
$r_1,\dots,r_t\in R$ such that
\begin{enumerate}
\item[\rm(i)]$lm(\textbf{f}_i)|lm(\textbf{f})$, $1\leq i\leq t$, i.e.,
$ind(lm(\textbf{f}_{i}))=ind(lm(\textbf{f}))$ and there exists
$x^{\alpha_{i}}\in Mon(A)$ such that
$\alpha_{i}+\exp(lm(\textbf{f}_i))=\exp(lm(\textbf{f}))$.
\item[\rm(ii)]$lc(\textbf{f})=r_1\sigma^{\alpha_1}(lc(\textbf{f}_1))c_{\alpha_1,\textbf{f}_1}+
\cdots+r_t\sigma^{\alpha_t}(lc(\textbf{f}_t))c_{\alpha_t,\textbf{f}_t}$,
with $c_{\alpha_i,\textbf{f}_i}:=$ \linebreak
$c_{\alpha_i,\exp(lm(\textbf{f}_i))}$.
\item[\rm (iii)]$\textbf{h}=\textbf{f}-\sum_{i=1}^tr_ix^{\alpha_i}\textbf{f}_i$.
\end{enumerate}
We say that $\textbf{f}$ reduces to $\textbf{h}$ by $F$, denoted
$\textbf{f}\xrightarrow{\,\, F\,\, }_{+} \textbf{h}$, if and only
if there exist vectors $\textbf{h}_1,\dots ,\textbf{h}_{t-1}\in
A^m$ such that
\begin{center}
$\begin{CD} \textbf{f} @>{F}>> \textbf{h}_1 @>{F}>> \textbf{h}_2
@>{F}>>\cdots @>{F}>>\textbf{h}_{t-1} @>{F}>>\textbf{h}\,.
\end{CD}$
\end{center}
$\textbf{f}$ is reduced {\rm(}also called minimal{\rm)} w.r.t. $F$
if $\textbf{f} = \textbf{\emph{0}}$ or there is no one step
reduction  of $\textbf{f}$ by $F$, i.e., one of the first two
conditions of Definition \ref{reductionformodules} fails.
Otherwise, we will say that $\textbf{f}$ is reducible w.r.t. $F$.
If $\textbf{f}\xrightarrow{\,\, F\,\, }_{+} \textbf{h}$ and
$\textbf{h}$ is reduced w.r.t. $F$, then we say that $\textbf{h}$
is a remainder for $\textbf{f}$ w.r.t. $F$.
\end{definition}
\begin{remark}
Related to the previous definition we have the following remarks:

(i) By Theorem \ref{coefficientes}, the coefficients
$c_{\alpha_i,\textbf{\emph{f}}_i}$ are unique and satisfy
\begin{center}
$x^{\alpha_i}x^{\exp(lm(\textbf{\emph{f}}_i))}=c_{\alpha_i,\textbf{\emph{f}}_i}x^{\alpha_i+\exp(lm(\textbf{\emph{f}}_i))}+p_{\alpha_i,\textbf{\emph{f}}_i}$,
\end{center}
where $p_{\alpha_i,\textbf{\emph{f}}_i}=0$ or
$\deg(lm(p_{\alpha_i,\textbf{\emph{f}}_i}))<|\alpha_i+\exp(lm(\textbf{\emph{f}}_i))|$,
$1\leq i\leq t$.

(ii) $lm(\textbf{\emph{f}})\succ lm(\textbf{\emph{h}})$ and
$\textbf{\emph{f}}-\textbf{\emph{h}}\in \langle F\rangle$, where
$\langle F\rangle$ is the submodule of $A^m$ generated by $F$.

(iii) The remainder of $\textbf{\emph{f}}$ is not unique.

(iv) By definition we will assume that $\textbf{0}\xrightarrow {F}
\textbf{0}$.

(v)
\[lt(\textbf{\emph{f}})=\sum_{i=1}^{t}r_{i}lt(x^{\alpha_{i}}lt(\textbf{\emph{f}}_{i})),\]
\end{remark}
The proofs of the next technical proposition and theorem can be
also adapted from \cite{Gallego2}.
\begin{proposition}\label{xtetafformodules}
Let $A$ be a $\sigma$-$PBW$ extension such that $c_{\alpha,\beta}$
is invertible for each $\alpha,\beta \in \mathbb{N}^n$. Let
$\textbf{f},\textbf{h}\in A^m$, $\theta \in \mathbb{N}^n$ and
$F=\{\textbf{f}_1,\dots ,\textbf{f}_t\}$ be a finite set of
non-zero vectors of $A^m$. Then,
\begin{enumerate}
\item[\rm{(i)}]If $\textbf{f}\xrightarrow{\,\, F\,\, } \textbf{h}$, then there exists $\textbf{p}\in
A^m$ with $\textbf{p}=\textbf{\emph{0}}$ or
$lm(x^{\theta}\textbf{f})\succ lm(\textbf{p})$ such that
$x^{\theta}\textbf{f}+\textbf{p}\xrightarrow{\,\, F\,\, }
x^{\theta}\textbf{h}$. In particular, if $A$ is quasi-commutative,
then $\textbf{p}=\textbf{\emph{0}}$.
\item[\rm{(ii)}]If $\textbf{f}\xrightarrow{\,\, F\,\, }_+
\textbf{h}$ and $\textbf{p}\in A^m$ is such that
$\textbf{p}=\textbf{\emph{0}}$ or $lm(\textbf{h})\succ
lm(\textbf{p})$, then $\textbf{f}+\textbf{p}\xrightarrow{\,\,
F\,\, }_+ \textbf{h}+\textbf{p}$.
\item[\rm{(iii)}]If $\textbf{f}\xrightarrow{\,\, F\,\, }_+ \textbf{h}$, then there
exists $\textbf{p}\in A^m$ with $\textbf{p}=\textbf{\emph{0}}$ or
$lm(x^{\theta}\textbf{f})\succ lm(\textbf{p})$ such that
$x^{\theta}\textbf{f}+\textbf{p}\xrightarrow{\,\, F\,\, }_+
x^{\theta}\textbf{h}$. If $A$ is quasi-commutative, then
$\textbf{p}=\textbf{\emph{0}}$.
\item[\rm{(iv)}]If $\textbf{f}\xrightarrow{\,\, F\,\, }_+ \textbf{\emph{0}}$, then there
exists $\textbf{p}\in A^m$ with $\textbf{p}=\textbf{\emph{0}}$ or
$lm(x^{\theta}\textbf{f})\succ lm(\textbf{p})$ such that
$x^{\theta}\textbf{f}+\textbf{p}\xrightarrow{\,\, F\,\, }_+
\textbf{\emph{0}}$. If $A$ is quasi-commutative, then
$\textbf{p}=\textbf{\emph{0}}$.
\end{enumerate}
\end{proposition}
\begin{theorem}\label{algdivformodules}
Let $F=\{\textbf{f}_1,\dots ,\textbf{f}_t\}$ be a set of non-zero
vectors of $A^m$ and $\textbf{f}\in A^m$, then the Division
Algorithm below produces polynomials $q_1,\dots ,q_t\in A$ and a
reduced vector $\textbf{h}\in A^m$ w.r.t. $F$ such that
$\textbf{f}\xrightarrow{\,\, F\,\, }_{+} \textbf{h}$ and
\begin{center}
$\textbf{f}=q_1\textbf{f}_1+\cdots +q_t\textbf{f}_t+\textbf{h}$
\end{center}
with
\begin{center}
$lm(\textbf{f})=\max\{lm(lm(q_1)lm(\textbf{f}_1)),\dots
,lm(lm(q_t)lm(\textbf{f}_t)),lm(\textbf{h})\}$.
\end{center}
\begin{center}
\fbox{\parbox[c]{11cm}{
\begin{center}
{\rm \textbf{Division Algorithm in} $A^m$}
\end{center}
\begin{description}
\item[]{\rm \textbf{INPUT}:} $\textbf{f},\textbf{f}_1,\dots ,\textbf{f}_t\in A^m \, \, \text{with}\,\, \textbf{f}_j\neq \textbf{\emph{0}}\, (1\leq j\leq t)$
\item[]{\rm \textbf{OUTPUT}:} $q_1,\dots ,q_t\in A\,\, ,\textbf{h}\in A^m \, \text{with}\,\, \textbf{f}=q_1\textbf{f}_1+\cdots
+q_t\textbf{f}_t+\textbf{h}$, $\textbf{h}$ reduced w.r.t. $\{\textbf{f}_1,\dots ,\textbf{f}_t\}$ and\\
$lm(\textbf{f})=\max\{lm(lm(q_1)lm(\textbf{f}_1)),\dots
,lm(lm(q_t)lm(\textbf{f}_t)),lm(\textbf{h})\}$
\item[]{\rm \textbf{INITIALIZATION}:} $q_1:=0,q_2:=0,\dots ,q_t:=0,\textbf{h}:=\textbf{f}$
\item[]{\rm \textbf{WHILE}} $\textbf{h}\neq \textbf{\emph{0}}$
and there exists $j$ such that $lm(\textbf{f}_j)$ divides
$lm(\textbf{h})$ \textbf{\emph{DO}}
\begin{quote}Calculate $J:=\{j\,|\,lm(\textbf{f}_j)$ divides $lm(\textbf{h})\}$

\smallskip
{\rm \textbf{FOR}} $j\in J$ {\rm \textbf{DO}}

\smallskip
\begin{quote}
Calculate $\alpha_j\in \mathbb{N}^n$ such that
$\alpha_j+\exp(lm(\textbf{f}_j))=\exp(lm(\textbf{h}))$
\end{quote}

\smallskip
{\rm \textbf{IF}} the equation $lc(\textbf{h})=\sum_{j\in
J}r_j\sigma^{\alpha_j}(lc(\textbf{f}_j))c_{\alpha_j,\textbf{f}_j}$
is soluble, where $c_{\alpha_j,\textbf{f}_j}$ are defined as in
Definition \ref{reductionformodules}

\smallskip
{\rm \textbf{THEN}}
\begin{quote}Calculate one solution $(r_j)_{j\in J}$

\smallskip
$\textbf{h}:=\textbf{h}-\sum_{j\in J}r_jx^{\alpha_j}\textbf{f}_j$

\smallskip
{\rm \textbf{FOR}} $j\in J$ {\rm \textbf{DO}}
\begin{quote}$q_j:=q_j+r_jx^{\alpha_j}$\end{quote}
\end{quote}
\emph{\textbf{ELSE}}
\begin{quote}
Stop
\end{quote}
\end{quote}
\end{description}}}
\end{center}
\end{theorem}
\begin{example}
We consider the Heisenberg algebra,
$A:=h_1(2)=\sigma(\mathbb{Q})\langle x, y, z\rangle$, with deglex
order and $x>y>z$ in $Mon(A)$ and the TOPREV order in $Mon(A^3)$
with $\boldsymbol{e}_1 \succ \boldsymbol{e}_2 \succ
\boldsymbol{e}_3$. Let  $\boldsymbol{f}:= x^2yz\boldsymbol{e}_1+
y^2z\boldsymbol{e}_2 + xz\boldsymbol{e}_1 + z^2\boldsymbol{e}_3$,
$\boldsymbol{f}_1:= xz\boldsymbol{e}_1 + x\boldsymbol{e}_3+
y\boldsymbol{e}_2$ and $\boldsymbol{f}_2:= xy\boldsymbol{e}_1 +
z\boldsymbol{e}_2 + z\boldsymbol{e}_3$. Following the Division
Algorithm we will compute $q_1, q_2\in A$ and $\boldsymbol{h}\in
A^3$ such that
$\boldsymbol{f}=q_1\boldsymbol{f}_1+q_2\boldsymbol{f}_2+\boldsymbol{h}$,
with $lm(\boldsymbol{f})=\max\{lm(lm(q_1)lm(\boldsymbol{f}_1)),
lm(lm(q_2)lm(\boldsymbol{f}_2)), lm(\boldsymbol{h})\}$. We will
represent the elements of $Mon(A)$ by $t^{\alpha}$ instead of
$x^{\alpha}$. For $j=1,2$, we will note
$\alpha_j:=(\alpha_{j1},\alpha_{j2},\alpha_{j3})\in \mathbb{N}^3$.

\textit{Step 1}: we start with $\boldsymbol{h}:= \boldsymbol{f},
q_1:= 0$ and $q_2:= 0$; since $lm(\boldsymbol{f}_1)\mid
lm(\boldsymbol{h})$ and $lm(\boldsymbol{f}_2)\mid
lm(\boldsymbol{h})$, we compute $\alpha_j$ such that $\alpha_j +
\exp(lm(\boldsymbol{f_j}))= \exp(lm(\boldsymbol{h}))$.

\noindent $\centerdot$ $lm(t^{\alpha_1} lm(\boldsymbol{f}_1))=
lm(\boldsymbol{h})$, so
$lm(x^{\alpha_{11}}y^{\alpha_{12}}z^{\alpha_{13}}xz)= x^2yz$, and
hence $\alpha_{11} = 1$; $\alpha_{12} = 1$; $\alpha_{13} = 0$.
Thus, $t^{\alpha_1} = xy$.

\noindent $\centerdot$ $lm(t^{\alpha_2} lm(\boldsymbol{f}_2))=
lm(\boldsymbol{h})$, so
$lm(x^{\alpha_{21}}y^{\alpha_{22}}z^{\alpha_{23}}xy)= x^2yz$, and
hence $\alpha_{21} = 1$; $\alpha_{22} = 0$; $\alpha_{23} = 1$.
Thus, $t^{\alpha_2} = xz$.

Next, for $j = 1, 2$ we compute $c_{\alpha_j,\boldsymbol{f_j}}$:

\noindent $\centerdot$ $t^{\alpha_1} t^{exp(lm(\boldsymbol{f}_1))}
= (xy)(xz) = x(2xy)z = 2x^2yz$. Thus,
$c_{\alpha_1,\boldsymbol{f}_1} = 2$.

\noindent $\centerdot$ $t^{\alpha_2} t^{exp(lm(\boldsymbol{f}_2))}
= (xz)(xy) = x(\frac{1}{2} xz + y)y = \frac{1}{2}x^2zy + xy^2 =
x^2yz + xy^2$. So, $c_{\alpha_2,\boldsymbol{f}_2} = 1$.

We must solve the equation
\begin{align*}
1 = lc(\boldsymbol{h}) &= r_1\sigma^{\alpha_1}(lc(\boldsymbol{f}_1))c_{\alpha_1,\boldsymbol{f}_1} + r_2\sigma^{\alpha_2}(lc(\boldsymbol{f}_2))c_{\alpha_2,\boldsymbol{f}_2}\\
&= r_1\sigma^{\alpha_1}(1)2 + r_2\sigma^{\alpha_2}(1)1\\
&= 2r_1 + r_2,
\end{align*}
then $r_1 = 0$ and $r_2 = 1$.

We make $\boldsymbol{h} := \boldsymbol{h} -
(r_1t^{\alpha_1}\boldsymbol{f}_1 +
r_2t^{\alpha_2}\boldsymbol{f}_2)$, i.e.,
\begin{align*}
\boldsymbol{h} &:= \boldsymbol{h} - (xz(xy\boldsymbol{e}_1 + z\boldsymbol{e}_2 + z\boldsymbol{e}_3))\\
&= \boldsymbol{h} - (xzxy\boldsymbol{e}_1 + xz^2\boldsymbol{e}_2 + xz^2\boldsymbol{e}_3)\\
&= \boldsymbol{h} - ((x^2yz + xy^2)\boldsymbol{e}_1 + xz^2\boldsymbol{e}_2 + xz^2\boldsymbol{e}_3)\\
&= x^2yz\boldsymbol{e}_1 + xz\boldsymbol{e}_1 + y^2z\boldsymbol{e}_2 + z^2\boldsymbol{e}_3 - x^2yz\boldsymbol{e}_1 - xy^2\boldsymbol{e}_1 - xz^2\boldsymbol{e}_2 - xz^2\boldsymbol{e}_3\\
&=- xy^2\boldsymbol{e}_1- xz^2\boldsymbol{e}_2-
xz^2\boldsymbol{e}_3+y^2z\boldsymbol{e}_2+xz\boldsymbol{e}_1+
z^2\boldsymbol{e}_3.
\end{align*}
In addition, we have $q_1 := q_1 + r_1t^{\alpha_1} = 0$ and $q_2
:= q_2 + r_2t^{\alpha_2} = xz$.

\textit{Step 2}: $\boldsymbol{h}:= - xy^2\boldsymbol{e}_1-
xz^2\boldsymbol{e}_2-
xz^2\boldsymbol{e}_3+y^2z\boldsymbol{e}_2+xz\boldsymbol{e}_1+
z^2\boldsymbol{e}_3$, so $lm(\boldsymbol{h}) =
xy^2\boldsymbol{e}_1$ and $lc(\boldsymbol{h}) = -1$; moreover,
$q_1 = 0$ and $q_2 = xz$. Since $lm(\boldsymbol{f}_2)\mid
lm(\boldsymbol{h})$, we compute $\alpha_2$ such that $\alpha_2 +
exp(lm(\boldsymbol{f}_2))= exp(lm(\boldsymbol{h}))$:

\noindent $\centerdot$ $lm(t^{\alpha_2} lm(\boldsymbol{f}_2))=
lm(\boldsymbol{h})$, then
$lm(x^{\alpha_{21}}y^{\alpha_{22}}z^{\alpha_{23}}xy)= xy^2$, so
$\alpha_{21} = 0$; $\alpha_{22} = 1$; $\alpha_{23} = 0$. Thus,
$t^{\alpha_2} = y$.

We compute $c_{\alpha_2,\boldsymbol{f}_2}$: $t^{\alpha_2}
t^{exp(lm(\boldsymbol{f}_2))} = y(xy) = 2xy^2$. Then,
$c_{\alpha_2,\boldsymbol{f}_2} = 2$.

We solve the equation
\begin{align*}
- 1 = lc(\boldsymbol{h}) &= r_2\sigma^{\alpha_2}(lc(\boldsymbol{f}_2))c_{\alpha_2,\boldsymbol{f}_2}\\
&= r_2\sigma^{\alpha_2}(1)2 = 2r_2,
\end{align*}
thus, $r_2 = -\frac{1}{2}$.

We make $\boldsymbol{h} := \boldsymbol{h} -
r_2t^{\alpha_2}\boldsymbol{f}_2$, i.e.,
\begin{align*}
\boldsymbol{h} &:= \boldsymbol{h} + \frac{1}{2}y(xy\boldsymbol{e}_1 + z\boldsymbol{e}_2 + z\boldsymbol{e}_3)\\
&= \boldsymbol{h} + \frac{1}{2}yxy\boldsymbol{e}_1 + \frac{1}{2}yz\boldsymbol{e}_2 + \frac{1}{2}yz\boldsymbol{e}_3\\
&=- xz^2\boldsymbol{e}_2- xz^2\boldsymbol{e}_3+
y^2z\boldsymbol{e}_2+ xz\boldsymbol{e}_1+
\frac{1}{2}yz\boldsymbol{e}_2+ \frac{1}{2}yz\boldsymbol{e}_3+
z^2\boldsymbol{e}_3.
\end{align*}
We have also that $q_1 := 0$ and $q_2 := q_2 + r_2t^{\alpha_2} =
xz - \frac{1}{2}y$.

\textit{Step 3}: $\boldsymbol{h} = - xz^2\boldsymbol{e}_2-
xz^2\boldsymbol{e}_3+ y^2z\boldsymbol{e}_2+ xz\boldsymbol{e}_1+
\frac{1}{2}yz\boldsymbol{e}_2+ \frac{1}{2}yz\boldsymbol{e}_3+
z^2\boldsymbol{e}_3$, so $lm(\boldsymbol{h}) =
xz^2\boldsymbol{e}_2$ and  $lc(\boldsymbol{h}) = -1$; moreover,
$q_1 = 0$ and $q_2 = xz - \frac{1}{2}y$. Since
$lm(\boldsymbol{f}_1)\nmid lm(\boldsymbol{h})$ and
$lm(\boldsymbol{f}_2)\nmid lm(\boldsymbol{h})$, then
$\boldsymbol{h}$ is reduced with respect to $\{\boldsymbol{f}_1,
\boldsymbol{f}_2\}$, so the algorithm stops.

Thus, we get $q_1, q_2 \in A$ and $\boldsymbol{h} \in A^3$ reduced
such that $\boldsymbol{f} = q_1\boldsymbol{f}_1 +
q_2\boldsymbol{f}_2 + \boldsymbol{h}$. In fact,
\begin{align*}
&q_1\boldsymbol{f}_1 + q_2\boldsymbol{f}_2 + \boldsymbol{h} = 0\boldsymbol{f}_1 + \left(xz - \frac{1}{2}y \right)\boldsymbol{f}_2 + \boldsymbol{h} \\
&= (xz - \frac{1}{2}y)(xy\boldsymbol{e}_1 + z\boldsymbol{e}_2 +
z\boldsymbol{e}_3) - xz^2\boldsymbol{e}_2- xz^2\boldsymbol{e}_3+
y^2z\boldsymbol{e}_2+ xz\boldsymbol{e}_1\\
&\hspace{.4 cm} +\frac{1}{2}yz\boldsymbol{e}_2+
\frac{1}{2}yz\boldsymbol{e}_3+
z^2\boldsymbol{e}_3\\
&= x^2yz\boldsymbol{e}_1 + xy^2\boldsymbol{e}_1 -
xy^2\boldsymbol{e}_1 + xz^2\boldsymbol{e}_2 -
\frac{1}{2}yz\boldsymbol{e}_2 +
xz^2\boldsymbol{e}_3 - \frac{1}{2}yz\boldsymbol{e}_3\\
&\hspace{.4 cm} - xz^2\boldsymbol{e}_2- xz^2\boldsymbol{e}_3+
y^2z\boldsymbol{e}_2+ xz\boldsymbol{e}_1+
\frac{1}{2}yz\boldsymbol{e}_2+ \frac{1}{2}yz\boldsymbol{e}_3+
z^2\boldsymbol{e}_3\\
&= x^2yz\boldsymbol{e}_1 + y^2z\boldsymbol{e}_2+
xz\boldsymbol{e}_1+ z^2\boldsymbol{e}_3 = \boldsymbol{f},
\end{align*}
and $\max\{lm(lm(q_i)lm(\boldsymbol{f}_i)),
lm(\boldsymbol{h})\}_{i=1,2} = \max\{0, x^2yz\boldsymbol{e}_1,
xz^2\boldsymbol{e}_2\} $ $= x^2yz\boldsymbol{e}_1 =
lm(\boldsymbol{f})$.
\end{example}

\section{Gröbner bases}
Our next purpose is to define Gröbner bases for submodules of
$A^m$.
\begin{definition}
Let $M\neq 0$ be a submodule of $A^m$ and let $G$ be a non empty
finite subset of non-zero vectors of $M$, we say that $G$ is a
Gröbner basis for $M$ if each element $\textbf{\emph{0}}\neq
\textbf{f}\in M$ is reducible w.r.t. $G$.
\end{definition}
We will say that $\{\textbf{0}\}$ is a Gröbner basis for $M=0$.
\begin{theorem}\label{teogrobnersigmapbwformodules}
Let $M\neq 0$ be a submodule of $A^m$ and let $G$ be a finite
subset of non-zero vectors of $M$. Then the following conditions
are equivalent:
\begin{enumerate}
\item[\rm(i)]$G$ is a Gröbner basis for $M$.
\item[\rm(ii)]For any vector $\textbf{f}\in A^m$,
\begin{center}
$\textbf{f}\in M$ if and only if $\textbf{f}\xrightarrow{\,\,
G\,\, }_{+} \textbf{\emph{0}}$.
\end{center}
\item[\rm(iii)]For any $\textbf{\emph{0}}\neq \textbf{f}\in M$ there exist $\textbf{g}_1,\dots ,\textbf{g}_t\in
G$ such that $lm(\textbf{g}_j)|lm(\textbf{f})$, $1\leq j\leq t$,
{\rm(}i.e., $ind(lm(\textbf{g}_{j}))=ind(lm(\textbf{f}))$ and
there exist $\alpha_j\in \mathbb{N}^n$ such that
$\alpha_j+\exp(lm(\textbf{g}_j))=\exp(lm(\textbf{f}))${\rm)} and
\begin{center}
$lc(\textbf{f})\in \langle
\sigma^{\alpha_1}(lc(\textbf{g}_1))c_{\alpha_1,\textbf{g}_1},\dots
,\sigma^{\alpha_t}(lc(\textbf{g}_t))c_{\alpha_t,\textbf{g}_t}\rangle$.
\end{center}
\item[\rm(iv)]For $\alpha \in \mathbb{N}^n$ and $1\leq u\leq m$, let $\langle \alpha ,M\rangle_u$ be the
left ideal of $R$ defined by
\begin{center}
$\langle \alpha ,M\rangle_u:=\langle lc(\textbf{f})|\textbf{f}\in
M,ind(lm(\textbf{f}))=u, \exp(lm(\textbf{f}))=\alpha\rangle$.
\end{center}
Then, $\langle \alpha ,M\rangle_u=J_u$, with
\begin{center}
$J_u:=\langle
\sigma^{\beta}(lc(\textbf{g}))c_{\beta,\textbf{g}}|\textbf{g}\in
G, ind(lm(\textbf{g}))=u\, \ \text{and} \ \beta +
\exp(lm(\textbf{g}))=\alpha\rangle$.
\end{center}
\end{enumerate}
\end{theorem}
\begin{proof}
(i)$\Rightarrow$ (ii): let $\textbf{\emph{f}}\in M$, if
$\textbf{\emph{f}}=\textbf{0}$, then by definition
$\textbf{\emph{f}}\xrightarrow{\,\, G\,\, }_{+} \textbf{0}$. If
$\textbf{\emph{f}}\neq \textbf{0}$, then there exists
$\textbf{\emph{h}}_1\in A^m$ such that
$\textbf{\emph{f}}\xrightarrow{\,\, G\,\, } \textbf{\emph{h}}_1$,
with $lm(\textbf{\emph{f}})\succ lm(\textbf{\emph{h}}_1)$ and
$\textbf{\emph{f}}-\textbf{\emph{h}}_1\in \langle
G\rangle\subseteq M$, hence $\textbf{\emph{h}}_1\in M$; if
$\textbf{\emph{h}}_1=\textbf{0}$, so we end. If
$\textbf{\emph{h}}_1\neq \textbf{0}$, then we can repeat this
reasoning for $\textbf{\emph{h}}_1$, and since $Mon(A^m)$ is well
ordered, we get that $\textbf{\emph{f}}\xrightarrow{\,\, G\,\,
}_{+} \textbf{0}$.

Conversely, if $\textbf{\emph{f}}\xrightarrow{\,\, G\,\, }_{+}
\textbf{0}$, then by Theorem \ref{algdivformodules}, there exist
$\textbf{\emph{g}}_1,\dots,\textbf{\emph{g}}_t\in G$ and
$q_1,\dots,q_t\in A$ such that
$\textbf{\emph{f}}=q_1\textbf{\emph{g}}_1+\cdots
+q_t\textbf{\emph{g}}_t$, i.e., $\textbf{\emph{f}}\in M$.

(ii)$\Rightarrow$ (i): evident.

(i)$\Leftrightarrow$ (iii): this is a direct consequence of
Definition \ref{reductionformodules}.

(iii)$\Rightarrow$ (iv) Since $R$ is left Noetherian, there exist
$r_1,\dots ,r_s\in R$,
$\textbf{\emph{f}}_1,\dots,\textbf{\emph{f}}_l$ $\in M$ such that
$\langle \alpha ,M\rangle_u=\langle r_1,\dots, r_s\rangle$,
$ind(lm(\textbf{\emph{f}}_i))=u$ and
$\exp(lm(\textbf{\emph{f}}_i))=\alpha$ for each $1\leq i\leq l$,
with $\langle r_1,\dots, r_s\rangle\subseteq $ $\langle
lc(\textbf{\emph{f}}_1),\dots,lc(\textbf{\emph{f}}_l)\rangle$.
Then, $\langle
lc(\textbf{\emph{f}}_1),\dots,lc(\textbf{\emph{f}}_l)\rangle=$
$\langle \alpha ,M\rangle_u$. Let $r\in \langle \alpha
,M\rangle_u$, there exist $a_1,\dots,a_l\in R$ such that
$r=a_1lc(\textbf{\emph{f}}_1)+\cdots+a_llc(\textbf{\emph{f}}_l)$;
by (iii), for each $i$, $1\leq i\leq l$, there exist
$\textbf{\emph{g}}_{1i},\dots,\textbf{\emph{g}}_{t_ii}\in G$ and
$b_{ji}\in R$ such that
$lc(\textbf{\emph{f}}_i)=b_{1i}\sigma^{\alpha_{1i}}(lc(\textbf{\emph{g}}_{1i}))c_{\alpha_{1i},\textbf{\emph{g}}_{1i}}+\cdots
+b_{t_ii}\sigma^{\alpha_{t_ii}}(lc(\textbf{\emph{g}}_{t_ii}))c_{\alpha_{t_ii},\textbf{\emph{g}}_{t_ii}}$,
with
$u=ind(lm(\textbf{\emph{f}}_i))=ind(lm(\textbf{\emph{g}}_{ji}))$
and
$\exp(lm(\textbf{\emph{f}}_i))=\alpha_{ji}+\exp(lm(\textbf{\emph{g}}_{ji}))$,
thus $\langle \alpha ,M\rangle_u\subseteq J_u$. Conversely, if
$r\in J_u$, then
$r=b_{1}\sigma^{\beta_{1}}(lc(\textbf{\emph{g}}_{1}))c_{\beta_{1},\textbf{\emph{g}}_{1}}+\cdots
+b_{t}\sigma^{\beta_{t}}(lc(\textbf{\emph{g}}_{t}))c_{\beta_{t},\textbf{\emph{g}}_{t}}$,
with $b_i\in R$, $\beta_i\in \mathbb{N}^n$,
$\textbf{\emph{g}}_i\in G$ such that
$ind(lm(\textbf{\emph{g}}_i))=u$ and
$\beta_i+\exp(lm(\textbf{\emph{g}}_i))=\alpha$ for any $1\leq
i\leq t$. Note that $x^{\beta_i}\textbf{\emph{g}}_i\in M$,
$ind(lm(x^{\beta_i}\textbf{\emph{g}}_i))=u$,
$\exp(lm(x^{\beta_i}\textbf{\emph{g}}_i))=\alpha$,
$lc(x^{\beta_i}\textbf{\emph{g}}_i)=\sigma^{\beta_{i}}(lc(\textbf{\emph{g}}_{i}))c_{\beta_{i},\textbf{\emph{g}}_{i}}$,
for $1\leq i\leq t$, and
$r=b_1lc(x^{\beta_1}\textbf{\emph{g}}_1)+\cdots+b_tlc(x^{\beta_t}\textbf{\emph{g}}_t)$,
i.e., $r\in \langle \alpha ,M\rangle_u$.

(iv)$\Rightarrow$ (iii): let $\textbf{0}\neq \textbf{\emph{f}}\in
M$ and let $u=ind(lm(\textbf{\emph{f}}))$,
$\alpha=\exp(lm(\textbf{\emph{f}}))$, then
$lc(\textbf{\emph{f}})\in \langle \alpha,M\rangle_u$; by (iv)
$lc(\textbf{\emph{f}})=b_{1}\sigma^{\beta_{1}}(lc(\textbf{\emph{g}}_{1}))c_{\beta_{1},\textbf{\emph{g}}_{1}}+\cdots
+b_{t}\sigma^{\beta_{t}}(lc(\textbf{\emph{g}}_{t}))c_{\beta_{t},\textbf{\emph{g}}_{t}}$,
with $b_i\in R$, $\beta_i\in \mathbb{N}^n$,
$\textbf{\emph{g}}_i\in G$ such that
$u=ind(lm(\textbf{\emph{g}}_i))$ and
$\beta_i+\exp(lm(\textbf{\emph{g}}_i))=\alpha$ for any $1\leq
i\leq t$. From this we conclude that
$lm(\textbf{\emph{g}}_j)|lm(\textbf{\emph{f}})$, $1\leq j\leq t$.
\end{proof}
From this theorem we get the following consequences.
\begin{corollary}\label{1117}
Let $M\neq 0$ be a submodule of $A^m$. Then,
\begin{enumerate}
\item[\rm(i)]If $G$ is a Gröbner basis for $M$, then $M=\langle G\rangle$.
\item[\rm(ii)]Let $G$ be a Gröbner basis for $M$, if $\textbf{f}\in M$ and
$\textbf{f}\xrightarrow{\,\, G\,\, }_{+} \textbf{h}$, with
$\textbf{h}$ reduced w.r.t. $G$, then
$\textbf{h}=\textbf{\emph{0}}$.
\item[\rm(iii)]Let $G=\{\textbf{g}_1,\dots,\textbf{g}_t\}$ be a set of non-zero vectors of $M$ with $lc(\textbf{g}_i)=1$, for each $1\leq i\leq t$, such
that given $\textbf{r}\in M$ there exists $i$ such that
$lm(\textbf{g}_i)$ divides $lm(\textbf{r})$. Then, $G$ is a
Gröbner basis of $M$.
\end{enumerate}
\end{corollary}

\section{Computing Gröbner bases}

The following two theorems are the support for the Buchberger's
algorithm for computing Gröbner bases when $A$ is a
quasi-commutative bijective $\sigma-PBW$ extension The proofs of
these results are as in \cite{Gallego2}.
\begin{definition}\label{BFformodules}
Let $F := \{\textbf{g}_1,\dots,\textbf{g}_s\}\subseteq A^m$ such
that the least common multiple of $\{lm(\textbf{g}_1),\dots
,lm(\textbf{g}_s)\}$, denoted by $\textbf{X}_F$, is non-zero. Let
$\theta\in \mathbb{N}^n$, $\beta_i:=\exp(lm(\textbf{g}_i))$ and
$\gamma_i\in \mathbb{N}^n$ such that
$\gamma_i+\beta_i=\exp(\textbf{X}_F)$, $1\leq i\leq s$.
$B_{F,\theta}$ will denote a finite set of generators of
\begin{center}
$S_{F,\theta}:=Syz_R[\sigma^{\gamma_1+\theta}(lc(\textbf{g}_1))c_{\gamma_1+\theta,\beta_1}
\ \cdots \
\sigma^{\gamma_s+\theta}(lc(\textbf{g}_s))c_{\gamma_s+\theta,\beta_s})]$.
\end{center}
For $\theta=\textbf{\emph{0}}:=(0,\dots,0)$, $S_{F,\theta}$ will
be denoted by $S_F$ and $B_{F,\theta}$ by $B_F$.
\end{definition}
\begin{theorem}\label{partialresultformodules}
Let $M\neq 0$ be a submodule of $A^m$ and let $G$ be a finite
subset of non-zero generators of $M$. Then the following
conditions are equivalent:
\begin{enumerate}
\item[\rm (i)]$G$ is a Gröbner basis of $M$.
\item[\rm (ii)]For all $F:= \{\textbf{g}_1,\dots,\textbf{g}_s\}\subseteq G$, with $\textbf{X}_F\neq \textbf{\emph{0}}$, and for all $\theta\in
\mathbb{N}^n$ and any $(b_1,\dots,b_s)\in B_{F,\theta}$,
\begin{center}
$\sum_{i=1}^sb_ix^{\gamma_i+\theta}\textbf{g}_i\xrightarrow{\,\,
G\,\, }_+ 0$.
\end{center}
\end{enumerate}
In particular, if $G$ is a Gröbner basis of $M$ then for all $F:=
\{\textbf{g}_1,\dots,\textbf{g}_s\}\subseteq G$, with
$\textbf{X}_F\neq \textbf{\emph{0}}$, and any $(b_1,\dots,b_s)\in
B_{F}$,
\begin{center}
$\sum_{i=1}^sb_ix^{\gamma_i}\textbf{g}_i\xrightarrow{\,\, G\,\,
}_+ 0$.
\end{center}
\end{theorem}
\begin{theorem}\label{maintheoremsigmapbwformodules}
Let $A$ be a quasi-commutative bijective $\sigma-PBW$ extension.
Let $M\neq 0$ be a submodule of $A^m$ and let $G$ be a finite
subset of non-zero generators of $M$. Then the following
conditions are equivalent:
\begin{enumerate}
\item[\rm (i)]$G$ is a Gröbner basis of $M$.
\item[\rm (ii)]For all $F:= \{\textbf{g}_1,\dots,\textbf{g}_s\}\subseteq G$, with $\textbf{X}_F\neq \textbf{\emph{0}}$,
and any $(b_1,\dots,b_s)\in B_{F}$,
\begin{center}
$\sum_{i=1}^sb_ix^{\gamma_i}\textbf{g}_i\xrightarrow{\,\, G\,\,
}_+ \textbf{\emph{0}}$.
\end{center}
\end{enumerate}
\end{theorem}
\begin{corollary}\label{algorithmforquasiformodules}
Let $A$ be a quasi-commutative bijective $\sigma-PBW$ extension.
Let $F=\{\textbf{f}_1,\dots ,\textbf{f}_s\}$ be a set of non-zero
vectors of $A^m$. The algorithm below produces a Gröbner basis for
the submodule $\langle \textbf{f}_1,\dots ,\textbf{f}_s\rangle$
{\rm(}$P(X)$ denotes the set of subsets of the set $X${\rm)}:
\begin{center}
\fbox{\parbox[c]{11cm}{
\begin{center}
{\rm \textbf{Gröbner Basis Algorithm for Modules\\
over Quasi-Commutative Bijective $\sigma-PBW$ Extensions}}
\end{center}
\begin{description}
\item[]{\rm \textbf{INPUT}:} $F := \{\textbf{f}_1,\dots,\textbf{f}_s\}\subseteq A^m$,
$\textbf{f}_i\neq \textbf{\emph{0}}$, $1\leq i\leq s$
\item[]{\rm \textbf{OUTPUT}:} $G=\{\textbf{g}_1,\dots ,\textbf{g}_t\}$ a Gröbner basis for $\langle F\rangle$
\item[]{\rm \textbf{INITIALIZATION}:} $G:=\emptyset, G':=F$
\item[]{\rm \textbf{WHILE}} $G'\neq G$ {\rm \textbf{DO}}
\begin{quote}$D:=P(G')-P(G)$

\smallskip

$G:=G'$

\smallskip

{\rm \textbf{FOR}} each $S:=\{\textbf{g}_{i_1},\dots
,\textbf{g}_{i_k}\}\in D$, with $\textbf{X}_S\neq
\textbf{\emph{0}}$, {\rm \textbf{DO}}

\smallskip

\begin{quote}
Compute $B_S$

\smallskip
{\rm \textbf{FOR}} each $\textbf{b}=(b_1,\dots ,b_k)\in B_S$ {\rm
\textbf{DO}}
\begin{quote}Reduce $\sum_{j=1}^kb_jx^{\gamma_j}\textbf{g}_{i_j}\xrightarrow{\,\, G'\,\, }_+ \textbf{r}$,
with $\textbf{r}$ reduced with respect to $G'$ and $\gamma_j$
defined as in Definition \ref{BFformodules}
\begin{quote}{\rm \textbf{IF}} $\textbf{r}\neq \textbf{\emph{0}}$ {\rm \textbf{THEN}}
\begin{quote}$G':=G'\cup \{\textbf{r}\}$
\end{quote}
\end{quote}
\end{quote}
\end{quote}
\end{quote}
\end{description}}}
\end{center}
\end{corollary}
From Theorem \ref{hilbertbasisforgpbw} and the previous corollary
we get the following direct conclusion.
\begin{corollary}\label{existenceformodules}
Let $A$ be a quasi-commutative bijective $\sigma-PBW$ extension.
Then each submodule of $A^m$ has a Gröbner basis.
\end{corollary}
Now we will illustrate with an example the algorithm presented in
Corollary \ref{algorithmforquasiformodules}.
\begin{example}\label{231}
We will consider the multiplicative analogue of the Weyl algebra
\[A:=\mathcal{O}_3(\lambda_{21}, \lambda_{31}, \lambda_{32}) = \mathcal{O}_3\left(2, \frac{1}{2}, 3\right) = \sigma(\mathbb{Q}[x_1])\langle x_2, x_3 \rangle ,\]
hence we have the relations
\[x_2 x_1 = \lambda_{21}x_1 x_2 = 2x_1 x_2,  \ \text{so} \ \sigma_2(x_1) = 2 x_1 \ \text{and} \ \delta_2(x_1) = 0,\]
\[x_3 x_1 = \lambda_{31}x_1 x_3 = \frac{1}{2}x_1 x_3,  \ \text{so} \ \sigma_3(x_1) = \frac{1}{2} x_1 \ \text{and} \ \delta_3(x_1) = 0,\]
\[x_3 x_2 = \lambda_{32}x_2 x_3 = 3x_2 x_3,  \ \text{so} \ c_{2,3} = 3, \]
and for $r \in \mathbb{Q}$, $\sigma_2(r) = r=\sigma_3(r)$. We
choose in  $Mon(A)$ the deglex order with $x_2
> x_3$ and in $Mon(A^2)$ the TOPREV order with  $\boldsymbol{e}_1 \succ
\boldsymbol{e}_2$.

Let $\boldsymbol{f}_1 = x_1^2x_2^2\boldsymbol{e}_1 +
x_2x_3\boldsymbol{e}_2$, $lm(\boldsymbol{f}_1) =
x_2^2\boldsymbol{e}_1$ and $\boldsymbol{f}_2 =
2x_1x_2x_3\boldsymbol{e}_1 + x_2\boldsymbol{e}_2$,
$lm(\boldsymbol{f}_2) = x_2x_3\boldsymbol{e}_1$. We will construct
a Gröbner basis for the module $M := \langle \boldsymbol{f}_1,
\boldsymbol{f}_2\rangle$.

\textit{Step 1}: we start with $G := \emptyset$, $G' := \{\boldsymbol{f}_1, \boldsymbol{f}_2\}$. Since $G' \ne G$, we make\\
$D:= {\mathcal{P}}(G') - {\mathcal{P}}(G)= \{S_1, S_2, S_{1,
2}\}$, with $S_1:= \{\boldsymbol{f}_1\}, S_2:=
\{\boldsymbol{f}_2\}, S_{1, 2}:= \{\boldsymbol{f}_1,
\boldsymbol{f}_2\}$. We also make $G: = G'$, and for every $S \in
D$ such that $\boldsymbol{X}_S \ne \textbf{0}$ we compute $B_{S}$:

$\centerdot$ For $S_1$ we have
\[Syz_{\mathbb{Q}[x_1]}[\sigma^{\gamma_1}(lc(\boldsymbol{f}_1))c_{\gamma_1, \beta_1}],\]
where $\beta_1 =$ $\exp(lm(\boldsymbol{f}_1)) = (2, 0)$;
$\boldsymbol{X}_{S_1} = l.c.m. \{lm(\boldsymbol{f}_1)\} =
lm(\boldsymbol{f}_1) = x_2^2\boldsymbol{e}_1$;
$\exp(\boldsymbol{X}_{S_1}) = (2, 0)$; $\gamma_1 =$
$\exp(\boldsymbol{X}_{S_1}) - \beta_1$ = (0, 0);
$x^{\gamma_1}x^{\beta_1} = x_2^2$, so $c_{\gamma_1, \beta_1} = 1$.
Then,
\begin{align*}
\sigma^{\gamma_1}(lc(\boldsymbol{f}_1))c_{\gamma_1, \beta_1} &=
\sigma^{\gamma_1}(x_1^2) 1 = \sigma_2^0 \sigma_3^0(x_1^2)= x_1^2.
\end{align*}
Thus, $Syz_{\mathbb{Q}[x_1]}[x_1^2] = \{0\}$ and $B_{S_{1}} =
\{0\}$, i.e., we do not add any vector to $G'$.

$\centerdot$ For $S_2$ we have an identical situation.

$\centerdot$ For $S_{1, 2}$ we compute
\[Syz_{\mathbb{Q}[x_1]}[\sigma^{\gamma_1}(lc(\boldsymbol{f}_1))c_{\gamma_1, \beta_1} \hspace{0.3 cm} \sigma^{\gamma_2}
(lc(\boldsymbol{f}_2))c_{\gamma_2, \beta_2}],\] where $\beta_1 =$
$\exp(lm(\boldsymbol{f}_1)) = (2, 0)$ and $\beta_2 =$
$\exp(lm(\boldsymbol{f}_2)) = (1, 1)$;

$\boldsymbol{X}_{S_{1,2}} = l.c.m. \{lm(\boldsymbol{f}_1),
lm(\boldsymbol{f}_2)\} = l.c.m.(x_2^2\boldsymbol{e}_1,
x_2x_3\boldsymbol{e}_1) = x_2^2x_3\boldsymbol{e}_1$;

$\exp(\boldsymbol{X}_{S_{1,2}})= (2, 1)$; $\gamma_1 =$
$\exp(\boldsymbol{X}_{S_{1,2}}) - \beta_1$ = (0, 1) and $\gamma_2
=$ $\exp(\boldsymbol{X}_{S_{1,2}}) - \beta_2$ = (1, 0);
$x^{\gamma_1}x^{\beta_1} = x_3x_2^2 = 3x_2x_3x_2 = 9x_2^2x_3$, so
$c_{\gamma_1, \beta_1} = 9$; in a similar way
$x^{\gamma_2}x^{\beta_2} = x_2^2x_3$, i.e., $c_{\gamma_2, \beta_2}
= 1$. Then,
\begin{align*}
\sigma^{\gamma_1}(lc(\boldsymbol{f}_1))c_{\gamma_1, \beta_1} &=
\sigma^{\gamma_1}(x_1^2)9 = \sigma_2^0 \sigma_3(x_1^2)9 =
(\sigma_3(x_1)\sigma_3(x_1))9= \frac{9}{4}x_1^2
\end{align*}
and
\begin{align*}
\sigma^{\gamma_2}(lc(\boldsymbol{f}_2))c_{\gamma_2, \beta_2} &=
\sigma^{\gamma_2}(2x_1) 1 = \sigma_2 \sigma_3^0(2x_1) =
\sigma_2(2x_1)= 4x_1.
\end{align*}
Hence $Syz_{\mathbb{Q}[x_1]}[\frac{9}{4}x_1^2 \hspace{0.3 cm}
4x_1] = \{(b_1, b_2) \in \mathbb{Q}[x_1]^2 \mid
b_1(\frac{9}{4}x_1^2) + b_2(4x_1) = 0\}$ and $B_{S_{1,2}} = \{(4,
- \frac{9}{4}x_1)\}$. From this we get
\begin{align*}
4x^{\gamma_1}\boldsymbol{f}_1 - \frac{9}{4}x_1x^{\gamma_2}\boldsymbol{f}_2 &= 4x_3(x_1^2x_2^2\boldsymbol{e}_1 + x_2x_3\boldsymbol{e}_2) - \frac{9}{4}x_1x_2(2x_1x_2x_3\boldsymbol{e}_1 + x_2\boldsymbol{e}_2)\\
&= 4x_3x_1^2x_2^2\boldsymbol{e}_1 + 4x_3x_2x_3\boldsymbol{e}_2 - \frac{9}{4}x_1x_22x_1x_2x_3\boldsymbol{e}_1 - \frac{9}{4}x_1x_2^2\boldsymbol{e}_2\\
&= 9x_1^2x_2^2x_3\boldsymbol{e}_1 + 12x_2x_3^2\boldsymbol{e}_2 - 9x_1^2x_2^2x_3\boldsymbol{e}_1 - \frac{9}{4}x_1x_2^2\boldsymbol{e}_2 \\
&= 12x_2x_3^2\boldsymbol{e}_2 -
\frac{9}{4}x_1x_2^2\boldsymbol{e}_2 := \boldsymbol{f}_3,
\end{align*}
so $lm(\boldsymbol{f}_3) = x_2x_3^2\boldsymbol{e}_2$. We observe
that $\boldsymbol{f}_3$ is reduced with respect to $G'$. We make
$G' := G' \cup \{\boldsymbol{f}_3\}$, i.e., $G' =
\{\boldsymbol{f}_1, \boldsymbol{f}_2, \boldsymbol{f}_3\}$.

\textit{Step 2}: since $G =\{\boldsymbol{f}_1, \boldsymbol{f}_2\}
\ne G' = \{\boldsymbol{f}_1, \boldsymbol{f}_2,
\boldsymbol{f}_3\}$, we make $D:= {\mathcal{P}}(G') -
{\mathcal{P}}(G)$, i.e., $D:= \{S_3, S_{1, 3}, S_{2, 3}, S_{1, 2,
3}\}$, where $S_1:= \{\boldsymbol{f}_1\}, S_{1, 3}:=
\{\boldsymbol{f}_1, \boldsymbol{f}_3\}, S_{2, 3}:=
\{\boldsymbol{f}_2, \boldsymbol{f}_3\}, S_{1, 2, 3}:=
\{\boldsymbol{f}_1, \boldsymbol{f}_2, \boldsymbol{f}_3\}$. We make
$G := G'$, and for every $S \in D$ such that $\boldsymbol{X}_S \ne
\textbf{0}$ we must compute $B_{S}$. Since
$\boldsymbol{X}_{S_{1,3}} = \boldsymbol{X}_{S_{2,3}}=
\boldsymbol{X}_{S_{1,2,3}}= \textbf{0}$, we only need to consider
$S_3$.

$\centerdot$ We have to compute
\[Syz_{\mathbb{Q}[x_1]}[\sigma^{\gamma_3}(lc(\boldsymbol{f}_3))c_{\gamma_3, \beta_3}],\]
where $\beta_3 =$ $\exp(lm(\boldsymbol{f}_3)) = (1, 2)$;
$\boldsymbol{X}_{S_{3}}= l.c.m. \{lm(\boldsymbol{f}_3)\} =
lm(\boldsymbol{f}_3) = x_2x_3^2\boldsymbol{e}_2$;
$\exp(\boldsymbol{X}_{S_{3}}) = (1, 2)$; $\gamma_3 =$
$\exp(\boldsymbol{X}_{S_{3}}) - \beta_3$ = (0, 0);
$x^{\gamma_3}x^{\beta_3} = x_2x_3^2$, so $c_{\gamma_3, \beta_3} =
1$. Hence
\begin{align*}
\sigma^{\gamma_3}(lc(\boldsymbol{f}_3))c_{\gamma_3, \beta_3} &=
\sigma^{\gamma_3}(12) 1 = \sigma_2^0 \sigma_3^0(12) = 12,
\end{align*}
and $Syz_{\mathbb{Q}[x_1]}[12] = \{0\}$, i.e., $B_{S_{3}} =
\{0\}$. This means that we not add any vector to $G'$ and hence $G
= \{\boldsymbol{f}_1, \boldsymbol{f}_2, \boldsymbol{f}_3\}$ is a
Gröbner basis for $M$.
\end{example}

\section{Syzygy of a module}\label{31}
We present in this section a method for computing the syzygy
module of a submodule $M=\langle \textbf{\emph{f}}_1,\dots
,\textbf{\emph{f}}_s\rangle$ of $A^m$ using Gröbner bases. This
implies that we have a method for computing such bases. Thus, we
will assume that $A$ is a bijective quasi-commutative $\sigma$-PBW
extension.

Let $f$ be the canonical homomorphism defined by
\begin{align*}
A^s& \xrightarrow{f} A^m\\
\textbf{\emph{e}}_j& \mapsto \textbf{\emph{f}}_j
\end{align*}
where $\{\textbf{\emph{e}}_1,\dots,\textbf{\emph{e}}_s\}$ is the
canonical basis of $A^s$. Observe that $f$ can be represented by a
matrix, i.e., if $\textbf{\emph{f}}_j:=(f_{1j},\dots,f_{mj})^T$,
then the matrix of $f$ in the canonical bases of $A^s$ and $A^m$
is
\begin{center}
$F:=\begin{bmatrix}\textbf{\emph{f}}_1 & \cdots &
\textbf{\emph{f}}_s\end{bmatrix}=
\begin{bmatrix}
f_{11} & \cdots & f_{1s}\\
\vdots & & \vdots\\
f_{m1} & \cdots & f_{ms}
\end{bmatrix}\in M_{m\times s}(A)$.
\end{center}
Note that $Im(f)$ is the column module of $F$, i.e., the left
$A$-module generated by the columns of $F$:
\begin{center}
$Im(f)=\langle
f(\textbf{\emph{e}}_1),\dots,f(\textbf{\emph{e}}_s)\rangle=\langle
\textbf{\emph{f}}_1,\dots ,\textbf{\emph{f}}_s \rangle =\langle F
\rangle$.
\end{center}
Moreover, observe that if $\textbf{\emph{a}}:=(a_1,\dots,a_s)^T\in
A^s$, then
\begin{equation}
f(\textbf{\emph{a}})=(\textbf{\emph{a}}^TF^T)^T.
\end{equation}
In fact,
\begin{align*}
f(\textbf{\emph{a}})&=a_1f(\textbf{\emph{e}}_1)+\cdots+a_sf(\textbf{\emph{e}}_s)=a_1\textbf{\emph{f}}_1+\cdots+a_s\textbf{\emph{f}}_s\\
&=a_1\begin{bmatrix}f_{11}\\ \vdots \\
f_{m1}\end{bmatrix}+\cdots+a_s\begin{bmatrix}f_{1s}\\ \vdots \\
f_{ms}\end{bmatrix}\\
&=\begin{bmatrix}a_1f_{11}+\cdots +a_sf_{1s}\\
\vdots \\
a_1f_{m1}+\cdots +a_sf_{ms}
\end{bmatrix}\\
& =(\begin{bmatrix}a_1 & \cdots & a_s\end{bmatrix}
\begin{bmatrix}f_{11} & \cdots & f_{m1}\\
\vdots & & \vdots \\f_{1s} & \cdots & f_{ms}\end{bmatrix})^T\\
&=(\textbf{\emph{a}}^TF^T)^T.
\end{align*}

We recall that
\begin{center}
$Syz(\{\textbf{\emph{f}}_1,\dots
,\textbf{\emph{f}}_s\}):=\{\textbf{\emph{a}}:=(a_1,\dots,a_s)^T\in
A^s|a_1\textbf{\emph{f}}_1+\cdots+a_s\textbf{\emph{f}}_s=\textbf{0}\}$.
\end{center}
Note that
\begin{equation}
Syz(\{\textbf{\emph{f}}_1,\dots ,\textbf{\emph{f}}_s\})=\ker(f),
\end{equation}
but $Syz(\{\textbf{\emph{f}}_1,\dots ,\textbf{\emph{f}}_s\})\neq
\ker(F)$ since we have
\begin{equation}\label{313}
\textbf{\emph{a}}\in Syz(\{\textbf{\emph{f}}_1,\dots
,\textbf{\emph{f}}_s\})\Leftrightarrow
\textbf{\emph{a}}^TF^T=\textbf{0}.
\end{equation}
The modules of syzygies of $M$ and $F$ are defined by
\begin{equation}
Syz(M):=Syz(F):=Syz(\{\textbf{\emph{f}}_1,\dots
,\textbf{\emph{f}}_s\}).
\end{equation}
The generators of $Syz(F)$ can be disposed into a matrix, so
sometimes we will refer to $Syz(F)$ as a matrix. Thus, if $Syz(F)$
is generated by $r$ vectors, $\textbf{\emph{z}}_1, \dots,
\textbf{\emph{z}}_r$, then
\begin{center}
$Syz(F)=\langle \textbf{\emph{z}}_1, \dots,
\textbf{\emph{z}}_r\rangle$,
\end{center}
and we will use the following matrix notation
\begin{center}
$Syz(F):=Z(F):=\begin{bmatrix}\textbf{\emph{z}}_1 & \cdots &
\textbf{\emph{z}}_r
\end{bmatrix}=\begin{bmatrix}z_{11} & \cdots & z_{1r}\\
\vdots & & \vdots\\
z_{s1} & \cdots & z_{sr}\end{bmatrix}\in M_{s\times r}(A)$,
\end{center}
thus we have
\begin{equation}
Z(F)^TF^T=0.
\end{equation}
Let $G:=\{\textbf{\emph{g}}_1,\dots,\textbf{\emph{g}}_t\}$ be a
Gröbner basis of $M$, then from Division Algorithm and Corollary
\ref{1117}, there exist polynomials $q_{ij}\in A$, $1\leq i\leq
t$, $1\leq j\leq s$ such that
\begin{align*}
\textbf{\emph{f}}_1 = & \,q_{11}\textbf{\emph{g}}_1+\cdots +q_{t1}\textbf{\emph{g}}_t\\
\vdots &\\
\textbf{\emph{f}}_s = & \,q_{1s}\textbf{\emph{g}}_1+\cdots
+q_{ts}\textbf{\emph{g}}_t,
\end{align*}
i.e.,
\begin{equation}\label{316}
F^T=Q^TG^T,
\end{equation}
with
\begin{center}
$Q:=[q_{ij}]=\begin{bmatrix}q_{11} & \cdots &
q_{1s}\\
\vdots &  & \vdots\\
q_{t1} & \cdots & q_{ts}
\end{bmatrix}$, $G:=\begin{bmatrix}\textbf{\emph{g}}_1 & \cdots & \textbf{\emph{g}}_t\end{bmatrix}:=\begin{bmatrix}g_{11} & \cdots &
g_{1t}\\
\vdots &  & \vdots\\
g_{m1} & \cdots & g_{mt}
\end{bmatrix}$.
\end{center}
From (\ref{316}) we get
\begin{equation}\label{317}
Z(F)^TQ^TG^T=0.
\end{equation}
From the algorithm of Corollary \ref{algorithmforquasiformodules}
we observe that each element of $G$ can be expressed as an
$A$-linear combination of columns of $F$, i.e., there exists
polynomials $h_{ji}\in A$ such that
\begin{align*}
\textbf{\emph{g}}_1 = & \,h_{11}\textbf{\emph{f}}_1+\cdots +h_{s1}\textbf{\emph{f}}_s\\
\vdots &\\
\textbf{\emph{g}}_t = & \,h_{1t}\textbf{\emph{f}}_1+\cdots
+h_{st}\textbf{\emph{f}}_s,
\end{align*}
so we have
\begin{equation}\label{318}
G^T=H^TF^T,
\end{equation}
with
\begin{center}
$H:=[h_{ji}]=\begin{bmatrix}h_{11} & \cdots &
h_{1t}\\
\vdots &  & \vdots\\
h_{s1} & \cdots & h_{st}
\end{bmatrix}$.
\end{center}
The next theorem will prove that $Syz(F)$ can be calculated using
$Syz(G)$, and in turn, Lemma \ref{syzygiesforG} below will
establish that for quasi-commutative bijective $\sigma-PBW$
extensions, $Syz(G)$ can be computed using $Syz(L_G)$, where
\begin{center}
$L_G:=\begin{bmatrix}lt(\textbf{\emph{g}}_1) & \cdots &
lt(\textbf{\emph{g}}_t)\end{bmatrix}$.
\end{center}
Suppose that $Syz(L_G)$ is generated by $l$ elements,
\begin{equation}\label{319a}
Syz(L_G):=Z(L_G):=\begin{bmatrix}\textbf{\emph{z}}_1'' & \cdots &
\textbf{\emph{z}}_l''\end{bmatrix}=\begin{bmatrix}z_{11}''
& \cdots & z_{1l}''\\
\vdots & & \vdots\\
z_{t1}'' & \cdots & z_{tl}''
\end{bmatrix}.
\end{equation}
The proof of Lemma \ref{syzygiesforG} will show that $Syz(G)$ can
be generated also by $l$ elements, say,
$\textbf{\emph{z}}_1',\dots,\textbf{\emph{z}}_l'$, i.e.,
$Syz(G)=\langle
\textbf{\emph{z}}_1',\dots,\textbf{\emph{z}}_l'\rangle$; we write
\begin{center}
$Syz(G):=Z(G):=\begin{bmatrix}\textbf{\emph{z}}_1' & \cdots &
\textbf{\emph{z}}_l'
\end{bmatrix}=\begin{bmatrix}z_{11}' & \cdots & z_{1l}'\\
\vdots & & \vdots\\
z_{t1}' & \cdots & z_{tl}'\end{bmatrix}\in M_{t\times l}(A)$,
\end{center}
and hence
\begin{equation}\label{319}
Z(G)^TG^T=0.
\end{equation}
\begin{theorem}\label{syzygiesmainresult}
With the above notation, $Syz(F)$ coincides with the column module
of the extended matrix $\begin{bmatrix}(Z(G)^TH^T)^T &
I_s-(Q^TH^T)^T\end{bmatrix}$, i.e., in a matrix notation
\begin{equation}
Syz(F)=\begin{bmatrix}(Z(G)^TH^T)^T & I_s-(Q^TH^T)^T\end{bmatrix}.
\end{equation}
\end{theorem}
\begin{proof}
Let $\textbf{\emph{z}}:=(z_1,\dots, z_s)^T$ be one of generators
of $Syz(F)$, i.e., one of columns of $Z(F)$, then by (\ref{313})
$\textbf{\emph{z}}^TF^T=\textbf{0}$, and by (\ref{316}) we have
$\textbf{\emph{z}}^TQ^TG^T=\textbf{0}$. Let
$\textbf{\emph{u}}:=(\textbf{\emph{z}}^TQ^T)^T$, then
$\textbf{\emph{u}}\in Syz(G)$ and there exists polynomials
$w_1,\dots,w_l\in A$ such that
$\textbf{\emph{u}}=w_1\textbf{\emph{z}}_1'+\cdots
+w_l\textbf{\emph{z}}_l'$, i.e.,
$\textbf{\emph{u}}=(\textbf{\emph{w}}^TZ(G)^T)^T$, with
$\textbf{\emph{w}}:=(w_1,\dots,w_l)^T$. Then,
$\textbf{\emph{u}}^TH^T=(\textbf{\emph{w}}^TZ(G)^T)H^T$, i.e.,
$\textbf{\emph{z}}^TQ^TH^T=(\textbf{\emph{w}}^TZ(G)^T)H^T$ and
from this we have
\begin{align*}
\textbf{\emph{z}}^T&
=\textbf{\emph{z}}^TQ^TH^T+\textbf{\emph{z}}^T
-\textbf{\emph{z}}^TQ^TH^T\\
&
=\textbf{\emph{z}}^TQ^TH^T+\textbf{\emph{z}}^T(I_s-Q^TH^T)\\
& =(\textbf{\emph{w}}^TZ(G)^T)H^T+\textbf{\emph{z}}^T(I_s-Q^TH^T).
\end{align*}
From this can be checked that $\textbf{\emph{z}}\in \langle
\begin{bmatrix}(Z(G)^TH^T)^T &
I_s-(Q^TH^T)^T\end{bmatrix}\rangle$.

Conversely, from (\ref{318}) and (\ref{319}) we have
$(Z(G)^TH^T)F^T=Z(G)^T(H^TF^T)=Z(G)^TG^T=0$, but this means that
each column of $(Z(G)^TH^T)^T$ is in $Syz(F)$. In a similar way,
from (\ref{318}) and (\ref{316}) we get
$(I_s-Q^TH^T)F^T=F^T-Q^TH^TF^T=F^T-Q^TG^T=F^T-F^T=0$, i.e., each
column of $(I_s-Q^TH^T)^T$ is also in $Syz(F)$. This complete the
proof.
\end{proof}
Our next task is to compute $Syz(L_G)$. Let
$L=[c_1\textbf{\emph{X}}_1\cdots c_t\textbf{\emph{X}}_t]$ be a
matrix of size $m\times t$, where
$\textbf{\emph{X}}_1=X_1\textbf{\emph{e}}_{i_1},\dots
,\textbf{\emph{X}}_t=X_t\textbf{\emph{e}}_{i_t}$ are monomials of
$A^m$, $c_1,\dots, c_t\in A-\{0\}$ and $1\leq i_1,\dots ,i_t\leq
m$. We note that some indexes $i_1,\dots ,i_t$ could be equals.
\begin{definition}\label{314}
We say that a syzygy $\textbf{h}=(h_1,\dots, h_t)^T\in Syz(L)$ is
homogeneous of degree $\textbf{X}=X\textbf{e}_i$, where $X\in
Mon(A)$ and $1\leq i\leq m$, if
\begin{enumerate}
\item[\rm(i)]$h_j$ is a term, for each $1\leq j\leq t$.
\item[\rm(ii)]For each $1\leq j\leq t$, either $h_j=0$ or if $h_j\neq 0$ then $lm(lm(h_j)\textbf{X}_j)=\textbf{X}$.
\end{enumerate}
\end{definition}
\begin{proposition}\label{315}
Let $L$ be as above. For quasi-commutative $\sigma-PBW$
extensions, $Syz(L)$ has a finite generating set of homogeneous
syzygies.
\end{proposition}
\begin{proof}Since $A^t$ is a Noetherian module, $Syz(L)$ is a finitely generated submodule of $A^t$.
So, it is enough to prove that each generator
$\textbf{\emph{h}}=(h_1,\dots,h_t)^T$ of $Syz(L)$ is a finite sum
of homogeneous syzygies of $Syz(L)$. We have
$h_1c_1X_1\textbf{\emph{e}}_{i_1}+\cdots
+h_tc_tX_t\textbf{\emph{e}}_{i_t}=\textbf{0}$, and we can group
together summands according to equal canonical vectors such that
$\textbf{\emph{h}}$ can be expressed as a finite sum of syzygies
of $Syz(L)$. We observe that each of such syzygies have null
entries for those places $j$ where $\textbf{\emph{e}}_{i_j}$ does
not coincide with the canonical vector of its group. The idea is
to prove that each of such syzygies is a sum of homogeneous
syzygies of $Syz(L)$. But this means that we have reduced the
problem to Lemma 4.2.2 of \cite{Loustaunau}, where the canonical
vector is the same for all entries. We include the proof for
completeness.

So, let $\textbf{\emph{f}}=(f_1,\dots ,f_t)^T\in Syz(c_1X_1,\dots
,c_tX_t)$, then $f_1c_1X_1+\cdots +f_tc_tX_t=0$; we expand each
polynomial $f_j$ as a sum of $u$ terms (adding zero summands, if
it is necessary):
\begin{center}
$f_j=a_{1j}Y_1+\cdots +a_{uj}Y_u$,
\end{center} where $a_{lj}\in
R$ and $Y_1\succ Y_2\succ \cdots \succ Y_u\in Mon(A)$ are the
different monomials we found in $f_1,\dots ,f_t$, $1\leq j\leq t$.
Then,
\begin{center}
$(a_{11}Y_1+\cdots +a_{u1}Y_u)c_1X_1+\cdots +(a_{1t}Y_1+\cdots +
a_{ut}Y_u)c_tX_t=0$.
\end{center}
Since $A$ is quasi-commutative, the product of two terms is a
term, so in the previous relation we can assume that there are
$d\leq tu$ different monomials, $Z_1, \ldots, Z_d$. Hence,
completing with zero entries (if it is necessary), we can write
\begin{center}
$\textbf{\emph{f}}=(b_{11}Y_{11},\dots ,b_{1t}Y_{1t})^T+\cdots
+(b_{d1}Y_{d1},\dots ,b_{dt}Y_{dt})^T$,
\end{center}
where $(b_{k1}Y_{k1},\dots ,b_{kt}Y_{kt})^T\in Syz(c_1X_1,\dots
,c_tX_t)$ is homogeneous of degree $Z_k$, $1\leq k\leq d$.
\end{proof}
\begin{definition}
Let $\textbf{X}_1,\dots,\textbf{X}_t\in Mon(A^m)$ and let
$J\subseteq \{1,\dots ,t\}$. Let
\begin{center}
$\textbf{X}_J=lcm\{\textbf{X}_j|j\in J\}$.
\end{center}
We say that $J$ is saturated with respect to
$\{\textbf{X}_1,\dots,\textbf{X}_t\}$, if
\begin{center}
$\textbf{X}_j|\textbf{X}_J\Rightarrow j\in J$,
\end{center}
for any $j\in \{1,\dots ,t\}$. The saturation $J'$ of $J$ consists
of all $j\in \{1,\dots ,t\}$ such that
$\textbf{X}_j|\textbf{X}_J$.
\end{definition}
\begin{lemma}\label{syzforlt}
Let $L$ be as above. For quasi-commutative bijective $\sigma-PBW$
extensions, a homogeneous generating set for $Syz(L)$ is
\begin{center}
$\{\textbf{s}_v^J|J\subseteq \{1,\dots ,t\}\,\text{is saturated
with respect to}\, \{\textbf{X}_1,\dots,\textbf{X}_t\}\, , 1\leq
v\leq r_J\}$,
\end{center}
where
\[
\textbf{s}_v^J= \sum_{j \in J} b_{vj}^J x^{\gamma_j}\textbf{e}_j,
\]
with $\gamma_j\in \mathbb{N}^n$ such that
$\gamma_j+\beta_j=\exp(\textbf{X}_J)$,
$\beta_j=\exp(\textbf{X}_j)$, $j\in J$, and $\textbf{b}_v^{J}:=(
b_{vj}^J)_{j\in J}$, with $B^J:=\{\textbf{b}_1^J,\dots
,\textbf{b}_{r_J}^J\}$ is a set of generators for $Syz_R
[\sigma^{\gamma_j}(c_j)c_{\gamma_j, \beta_j} \mid j \in J]$.
\end{lemma}
\begin{proof}
First note that $\textbf{\emph{s}}_v^J$ is a homogeneous syzygy of
$Syz(L)$ of degree $\textbf{\emph{X}}_J$ since each entry of
$\textbf{\emph{s}}_v^J$ is a term, for each non-zero entry we have
$lm(x^{\gamma_j}\textbf{\emph{X}}_j)=\textbf{\emph{X}}_J$, and
moreover, if $i_J:=ind(\textbf{\emph{X}}_J)$, then
\begin{align*}
((\textbf{\emph{s}}_v^J)^TL^T)^T &= \sum_{j \in J} b_{vj}^J
x^{\gamma_j}c_j \textbf{\emph{X}}_j
= \sum_{j \in J} b_{vj}^J \sigma^{\gamma_j}(c_j)x^{\gamma_j} \textbf{\emph{X}}_j \\
&= (\sum_{j \in J} (b_{vj}^J \sigma^{\gamma_j}(c_j)c_{\gamma_j,
\beta_j})x^{\gamma_j + \beta_j})\textbf{\emph{e}}_{i_J}\\
&=\textbf{0}.
\end{align*}
On the other hand, let $\textbf{\emph{h}}\in Syz(L)$, then by
Proposition \ref{315}, $Syz(L)$ is generated by homogeneous
syzygies, so we can assume that $\textbf{\emph{h}}$ is a
homogeneous syzygy of some degree
$\textbf{\emph{Y}}=Y\textbf{\emph{e}}_i$, $Y:=x^{\alpha}$. We will
represent $\textbf{\emph{h}}$ as a linear combination of syzygies
of type $\textbf{\emph{s}}_v^J$. Let
$\textbf{\emph{h}}=(d_1Y_1,\dots ,d_tY_t)^T$, with $d_k\in R$ and
$Y_k:=x^{\alpha_k}$, $1\leq k\leq t$, let $J=\{j\in\{1,\dots
,t\}|d_j\neq 0\}$, then
$lm(Y_j\textbf{\emph{X}}_j)=\textbf{\emph{Y}}$ for $j\in J$, and
\begin{align*}
\boldsymbol{0} = \sum_{j \in J} d_jY_jc_j\textbf{\emph{X}}_j &=
\sum_{j \in J} d_j \sigma^{\alpha_j}(c_j)Y_j \textbf{\emph{X}}_j
= \sum_{j \in J} d_j \sigma^{\alpha_j}(c_j) c_{\alpha_j, \beta_j} \boldsymbol{Y}. \\
\end{align*}
In addition, since $lm(Y_j\textbf{\emph{X}}_j)=\boldsymbol{Y}$
then $\textbf{\emph{X}}_j \mid \boldsymbol{Y}$ for any $j \in J$,
and hence $\textbf{\emph{X}}_J \mid \boldsymbol{Y}$, i.e., there
exists $\theta$ such that $\theta +
\exp(\textbf{\emph{X}}_J)=\alpha=\theta+\gamma_j+\beta_j$; but,
$\alpha_j+\beta_j=\alpha$ since
$lm(Y_j\textbf{\emph{X}}_j)=\boldsymbol{Y}$, so $\alpha_j=\theta +
\gamma_j$.

Thus,
\begin{align*}
\boldsymbol{0} &= \sum_{j \in J} d_j \sigma^{\alpha_j}(c_j)
c_{\alpha_j, \beta_j} \boldsymbol{Y}= \sum_{j \in J} d_j
\sigma^{\theta + \gamma_j}(c_j) c_{\theta + \gamma_j, \beta_j}
\boldsymbol{Y},
\end{align*}
and from Remark \ref{identities} we get that
\begin{align*}
0 &= \sum_{j \in J} d_j \sigma^{\theta + \gamma_j}(c_j) c_{\theta
+ \gamma_j, \beta_j}
= \sum_{j \in J} d_j c_{\theta, \gamma_j}^{-1} c_{\theta, \gamma_j} \sigma^{\theta + \gamma_j}(c_j) c_{\theta + \gamma_j, \beta_j}\\
&= \sum_{j \in J} d_j c_{\theta, \gamma_j}^{-1}
\sigma^{\theta}(\sigma^{\gamma_j}(c_j)) c_{\theta, \gamma_j}
c_{\theta + \gamma_j, \beta_j}\\
& = \sum_{j \in J} d_j c_{\theta, \gamma_j}^{-1}
\sigma^{\theta}(\sigma^{\gamma_j}(c_j))
\sigma^{\theta}(c_{\gamma_j, \beta_j}) c_{\theta, \gamma_j +
\beta_j}.
\end{align*}
We multiply the last equality by $c_{\theta,
\exp(\textbf{\emph{X}}_J)}^{-1}$, but $c_{\theta,
\exp(\textbf{\emph{X}}_J)}^{-1}=c_{\theta, \gamma_j +
\beta_j}^{-1}$ for any $j \in J$, so
\[0 =  \sum_{j \in J} d_j c_{\theta, \gamma_j}^{-1} \sigma^{\theta}(\sigma^{\gamma_j}(c_j) c_{\gamma_j, \beta_j}).\]
Since $A$ is bijective, there exists $d_j'$ such that
$\sigma^\theta(d_j') = d_j c_{\theta, \gamma_j}^{-1}$, so
\[0 =  \sum_{j \in J} \sigma^\theta(d_j') \sigma^{\theta}(\sigma^{\gamma_j}(c_j) c_{\gamma_j, \beta_j}),\]
and from this we get
\[0 = \sum_{j \in J} d_j' \sigma^{\gamma_j}(c_j) c_{\gamma_j, \beta_j}.\]
Let $J'$ be the saturation of $J$ with respect to
$\{\textbf{\emph{X}}_1,\dots ,\textbf{\emph{X}}_t\}$, since $d_j =
0$ if $j \in J'- J$, then $d_j'=0$, and hence, $(d_j'\mid j \in
J') \in Syz_R [\sigma^{\gamma_j}(c_j)c_{\gamma_j, \beta_j} \mid j
\in J']$. From this we have
\[(d_j' \mid j \in J') = \sum_{v = 1}^{r_{J'}}a_v b_{vj}^{J'}.\]

Since $\textbf{\emph{X}}_{J'}=\textbf{\emph{X}}_J$, then
$\textbf{\emph{X}}_{J'}$ also divides $\boldsymbol{Y}$, and hence
\begin{align*}
\boldsymbol{h} &= \sum_{j = 1}^{t} d_j Y_j \boldsymbol{e}_j =
\sum_{j \in J'} d_j c_{\theta, \gamma_j}^{-1} x^{\theta}
x^{\gamma_j} \boldsymbol{e}_j = \sum_{j \in J'} \sigma^{\theta}(d_j') x^{\theta} x^{\gamma_j} \boldsymbol{e}_j\\
&= \sum_{j \in J'} x^{\theta} d_j' x^{\gamma_j} \boldsymbol{e}_j =
\sum_{j \in J'} x^{\theta}\left(\sum_{v =1}^{r_{J'}} a_v
b_{vj}^{J'} \right)x^{\gamma_j} \boldsymbol{e}_j = \sum_{j \in J'} \sum_{v =1}^{r_{J'}} x^{\theta}a_v b_{vj}^{J'} x^{\gamma_j} \boldsymbol{e}_j\\
&= \sum_{v =1}^{r_{J'}} x^{\theta}a_v \sum_{j \in J'} b_{vj}^{J'} x^{\gamma_j} \boldsymbol{e}_j\\
&= \sum_{v =1}^{r_{J'}}
\sigma^{\theta}(a_v)x^{\theta}\boldsymbol{s}_v^{J'}.
\end{align*}
\end{proof}
Finally, we will calculate $Syz(G)$ using $Syz(L_G)$. Applying
Division Algorithm and Corollary \ref{1117} to the columns of
$Syz(L_G)$ (see (\ref{319a})), for each $1\leq v\leq l$ there
exists polynomials $p_{1v},\dots,p_{tv}\in A$ such that
\begin{center}
$z_{1v}''\textbf{\emph{g}}_1+\cdots
+z_{tv}''\textbf{\emph{g}}_t=p_{1v}\textbf{\emph{g}}_1+\cdots
+p_{tv}\textbf{\emph{g}}_t$,
\end{center}
i.e.,
\begin{equation}\label{311}
Z(L_G)^TG^T=P^TG^T,
\end{equation}
with
\begin{center}
$P:=\begin{bmatrix}p_{11}
& \cdots & p_{1l}\\
\vdots & & \vdots\\
p_{t1} & \cdots & p_{tl}
\end{bmatrix}$.
\end{center}
With this notation, we have the following result.
\begin{lemma}\label{syzygiesforG}
For quasi-commutative bijective $\sigma-PBW$ extensions, the
column module of $Z(G)$ coincides with the column module of
$Z(L_G)-P$, i.e., in a matrix notation
\begin{equation}
Z(G)=Z(L_G)-P.
\end{equation}
\end{lemma}
\begin{proof}
From (\ref{311}), $(Z(L_G)-P)^TG^T=0$, so each column of
$Z(L_G)-P$ is in $Syz(G)$, i.e., each column of $Z(L_G)-P$ is an
$A$-linear combination of columns of $Z(G)$. Thus, $\langle
Z(L_G)-P \rangle \subseteq \langle Z(G)\rangle$.

Now we have to prove that $\langle Z(G)\rangle\subseteq \langle
Z(L_G)-P \rangle$. Suppose that $\langle Z(G)\rangle\nsubseteq
\langle Z(L_G)-P \rangle$, so there exists
$\textbf{\emph{z}}'=(z_1',\dots,z_t')^T\in \langle Z(G)\rangle$
such that $\textbf{\emph{z}}'\notin \langle Z(L_G)-P \rangle$;
from all such vectors we choose one such that
\begin{equation}\label{X1}
\textbf{\emph{X}}:=\max_{1\leq j\leq
t}\{lm(lm(z_j')lm(\textbf{\emph{g}}_j))\}
\end{equation}
be the least. Let $\textbf{\emph{X}}=X\textbf{\emph{e}}_i$ and
\begin{center}
$J:=\{j\in\{1,\dots
,t\}|lm(lm(z_j')lm(\textbf{\emph{g}}_j))=\textbf{\emph{X}}\}$.
\end{center}
Since $A$ is quasi-commutative and $\textbf{\emph{z}}'\in Syz(G)$
then
\[
\sum_{j\in J}lt(z_j')lt(\textbf{\emph{g}}_j)=\textbf{0}.
\]
Let $\textbf{\emph{h}}:=\sum_{j\in J
}lt(z_j')\widetilde{\textbf{\emph{e}}}_j$, where
$\widetilde{\textbf{\emph{e}}}_1,\dots
,\widetilde{\textbf{\emph{e}}}_t$ is the canonical basis of $A^t$.
Then, $\textbf{\emph{h}}\in Syz(lt(\textbf{\emph{g}}_1),\dots
,lt(\textbf{\emph{g}}_t))$ is a homogeneous syzygy of degree
$\textbf{\emph{X}}$. Let $B:=\{\textbf{\emph{z}}_1'',\dots,
\textbf{\emph{z}}_l''\}$ be a homogeneous generating set for the
syzygy module $Syz(L_G))$, where $\textbf{\emph{z}}_v''$ has
degree $\textbf{\emph{Z}}_v=Z_v\textbf{\emph{e}}_{i_v}$  (see
(\ref{319a})). Then,
$\textbf{\emph{h}}=\sum_{v=1}^la_v\textbf{\emph{z}}_v''$, where
$a_v\in A$, and hence
\begin{equation*}
\textbf{\emph{h}}=(a_1z_{11}''+\cdots +a_lz_{1l}'',\dots
,a_1z_{t1}''+\cdots +a_lz_{tl}'')^T.
\end{equation*}
We can assume that for each $1\leq v\leq l$, $a_v$ is a term. In
fact, consider the first entry of $\textbf{\emph{h}}$: completing
with null terms, each $a_v$ is an ordered sum of $s$ terms
\begin{center}
$(c_{11}X_{11}+\cdots +c_{1s}X_{1s})z_{11}''+\cdots
+(c_{l1}X_{l1}+\cdots +c_{ls}X_{ls})z_{1l}''$,
\end{center}
with $X_{v1}\succ X_{v2}\succ \cdots \succ X_{vs}$ for each $1\leq
v\leq l$, so
\begin{equation}\label{primeraentrada}
\left.
\begin{gathered}
lm(X_{11}lm(z_{11}''))\succ lm(X_{12}lm(z_{11}''))\succ\cdots \succ lm(X_{1s}lm(z_{11}'')) \\
\vdots \\
lm(X_{l1}lm(z_{1l}''))\succ lm(X_{l2}lm(z_{1l}''))\succ \cdots
\succ lm(X_{ls}lm(z_{1l}''))
\end{gathered}
\right\}
\end{equation}
Since each $\textbf{\emph{z}}_v''$ is a homogeneous syzygy, each
entry $z_{jv}''$ of $\textbf{\emph{z}}_v''$ is a term, but the
first entry of $\textbf{\emph{h}}$ is also a term, then from
(\ref{primeraentrada}) we can assume that $a_v$ is a term.

We note that for $j\in J$
\begin{equation*}
lt(z_j')=a_1z_{j1}''+\cdots +a_lz_{jl}'',
\end{equation*}
and for $j\notin J$
\begin{equation*}
a_1z_{j1}''+\cdots +a_lz_{jl}''=0.
\end{equation*}
Moreover, let $j\in J$, so $lm(lm(a_1z_{j1}''+\cdots
+a_lz_{jl}'')lm(\textbf{\emph{g}}_j))=lm(lm(z_j')lm(\textbf{\emph{g}}_j))=\textbf{\emph{X}}$,
and we can choose those $v$ such that $lm(a_vz_{jv}'')=lm(z_j')$,
for the others $v$ we can take $a_v=0$. Thus, for $j$ and such $v$
we have
\begin{equation*}
lm(lm(a_v)lm(lm(z_{jv}'')lm(\textbf{\emph{g}}_j)))=\textbf{\emph{X}}=X\textbf{\emph{e}}_i.
\end{equation*}
On the other hand, for $j,j'\in J$ with $j'\neq j$, we know that
$\textbf{\emph{z}}_v''$ is homogeneous of degree
$\textbf{\emph{Z}}_v=Z_v\textbf{\emph{e}}_{i_v}$, hence, if
$z_{j'v}''\neq 0$, then
$lm(lm(z_{j'v}'')lm(\textbf{\emph{g}}_{j'}))=\textbf{\emph{Z}}_v=lm(lm(z_{jv}'')lm(\textbf{\emph{g}}_{j}))$.
Thus, we must conclude that $i_v=i$ and
\begin{equation}\label{125}
lm(lm(a_v)lm(lm(z_{jv}'')lm(\textbf{\emph{g}}_j)))=\textbf{\emph{X}},
\end{equation}
for any $v$ and any $j$ such that $a_{v}\neq 0$ and $z_{jv}''\neq
0$.

We define $\textbf{\emph{q}}':=(q_1',\dots,q_t')^T$, where
$q_j':=z_j'$ if $j\notin J$ and $q_j':=z_j'-lt(z_j')$ if $j\in J$.
We observe that
$\textbf{\emph{z}}'=\textbf{\emph{h}}+\textbf{\emph{q}}'$, and
hence
$\textbf{\emph{z}}'=\sum_{v=1}^la_v\textbf{\emph{z}}_v''+\textbf{\emph{q}}'=
\sum_{v=1}^la_v(\textbf{\emph{s}}_v+\textbf{\emph{p}}_v)+\textbf{\emph{q}}'$,
with
$\textbf{\emph{s}}_v:=\textbf{\emph{z}}_v''-\textbf{\emph{p}}_v$,
where $\textbf{\emph{p}}_v$ is the column $v$ of matrix $P$
defined in (\ref{311}). Then, we define
\begin{center}
$\textbf{\emph{r}}:=(\sum_{v=1}^la_v\textbf{\emph{p}}_v)+\textbf{\emph{q}}'$,
\end{center}
and we note that
$\textbf{\emph{r}}=\textbf{\emph{z}}'-\sum_{v=1}^la_v\textbf{\emph{s}}_v\in
Syz(G)-\langle Z(L_G)-P\rangle$. We will get a contradiction
proving that $\max_{1\leq j\leq
t}\{lm(lm(r_j)lm(\textbf{\emph{g}}_j))\}\prec\textbf{\emph{X}}$.
For each $1\leq j\leq t$ we have
\begin{center}
$r_j=a_1p_{j1}+\cdots+a_lp_{jl}+q_j'$
\end{center}
and hence
\begin{align*}
lm(lm(r_j)lm(\textbf{\emph{g}}_j))& =lm(lm(a_1p_{j1}+\cdots+a_lp_{jl}+q_j')lm(\textbf{\emph{g}}_j))\\
& \preceq lm(\max\{lm(a_1p_{j1}+\cdots+a_lp_{jl}),lm(q_j')\}lm(\textbf{\emph{g}}_j))\\
& \preceq lm(\max\{\max_{1\leq v\leq
l}\{lm(lm(a_v)lm(p_{jv}))\},lm(q_j')\}lm(\textbf{\emph{g}}_j)).
\end{align*}
By the definition of $\textbf{\emph{q}}'$ we have that for each
$1\leq j\leq t$,
$lm(lm(q_j')lm(\textbf{\emph{g}}_j))\prec\textbf{\emph{X}}$. In
fact, if $j\notin J$,
$lm(lm(q_j')lm(\textbf{\emph{g}}_j))=lm(lm(z_j')lm(\textbf{\emph{g}}_j))\prec\textbf{\emph{X}}$,
and for $j\in J$,
$lm(lm(q_j')lm(\textbf{\emph{g}}_j))=lm(lm(z_j'-lt(z_j'))lm(\textbf{\emph{g}}_j))\prec\textbf{\emph{X}}$.
On the other hand,
\begin{equation*}
\sum_{j=1}^tz_{jv}''\textbf{\emph{g}}_j=\sum_{j=1}^tp_{jv}\textbf{\emph{g}}_j,
\end{equation*}
with
\[
lm(\sum_{j=1}^tz_{jv}''\textbf{\emph{g}}_j)=\max_{1\leq j\leq
t}\{lm(lm(p_{jv})lm(\textbf{\emph{g}}_j))\}.
\]
But, $\sum_{j=1}^tz_{jv}''lt(\textbf{\emph{g}}_j)=\textbf{0}$ for
each $v$, then
\[
lm(\sum_{j=1}^tz_{jv}''\textbf{\emph{g}}_j)\prec\max_{1\leq j\leq
t}\{lm(lm(z_{jv}'')lm(\textbf{\emph{g}}_j))\}.
\]
Hence,
\[
\max_{1\leq j\leq
t}\{lm(lm(p_{jv})lm(\textbf{\emph{g}}_j))\}\prec\max_{1\leq j\leq
t}\{lm(lm(z_{jv}'')lm(\textbf{\emph{g}}_j))\}
\]
for each $1\leq v\leq l$. From (\ref{125}), $
\max_{\substack{1\leq j\leq t \\ 1\leq v\leq
l}}\{lm(lm(a_v)lm(lm(p_{jv})lm(\textbf{\emph{g}}_j)))\}\prec\max_{\substack{1\leq
j\leq t \\ 1\leq v\leq l
}}\{lm(lm(a_v)lm(lm(z_{jv}'')lm(\textbf{\emph{g}}_j)))\}=\textbf{\emph{X}}
$, and hence, we can conclude that $\max_{1\leq j\leq
t}\{lm(lm(r_j)lm(\textbf{\emph{g}}_j))\}\prec\textbf{\emph{X}}$.
\end{proof}
\begin{example}\label{319b}
Let $M := \langle \boldsymbol{f}_1, \boldsymbol{f}_2\rangle$,
where $\boldsymbol{f}_1 = x_1^2x_2^2\boldsymbol{e}_1 +
x_2x_3\boldsymbol{e}_2$ and $\boldsymbol{f}_2 =
2x_1x_2x_3\boldsymbol{e}_1 + x_2\boldsymbol{e}_2\in A^2$, with $A
:= \sigma(\mathbb{Q}[x_1])\langle x_2, x_3 \rangle$. In Example
\ref{231} we computed a Gröbner basis $G = \{\boldsymbol{f}_1,
\boldsymbol{f}_2, \boldsymbol{f}_3\}$ of $M$, where
$\boldsymbol{f}_3 = 12x_2x_3^2\boldsymbol{e}_2 -
\frac{9}{4}x_1x_2^2\boldsymbol{e}_2$. Now we will calculate
$Syz(F)$ with $F = \{\boldsymbol{f}_1, \boldsymbol{f}_2\}$:
\begin{enumerate}
\item[(i)] Firstly we compute $Syz(L_G)$ using Lemma \ref{syzforlt}:
\[L_G :=
\begin{bmatrix}
lt(\boldsymbol{f}_{1})  &lt(\boldsymbol{f}_{2})
&lt(\boldsymbol{f}_{3})
\end{bmatrix} =
\begin{bmatrix}
x_1^2x_2^2\boldsymbol{e}_1  &2x_1x_2x_3\boldsymbol{e}_1
&12x_2x_3^2\boldsymbol{e}_2
\end{bmatrix}.
\]
For this we choose the saturated subsets $J$ of $\{1, 2, 3\}$ with
respect to $\{x_2^2\boldsymbol{e}_1,$ $ x_2x_3\boldsymbol{e}_1,
x_2x_3^2\boldsymbol{e}_2\}$
and such that $\boldsymbol{X}_J \ne 0$: \\
$\centerdot$ For $J_1 = \{1\}$ we compute a system $B^{J_1}$ of
generators of
\[Syz_{\mathbb{Q}[x_1]}[\sigma^{\gamma_1}(lc(\boldsymbol{f}_1))c_{\gamma_1, \beta_1}],\]
where $\beta_1 :=$ $\exp(lm(\boldsymbol{f}_1))$ and $\gamma_1 =$
$\exp(\boldsymbol{X}_{J_{1}}) - \beta_1$. Then, $B^{J_1} = \{0\}$,
and hence we have only one generator $\boldsymbol{b}_1^{J_1} =
(b_{11}^{J_1})=  0$ and $\boldsymbol{s}_1^{J_1} =
b_{11}^{J_1}x^{\gamma_1}\boldsymbol{\tilde{e}}_1 =
0\boldsymbol{\tilde{e}}_1$, with $\boldsymbol{\tilde{e}}_1 = (1,
0, 0)^T$.

$\centerdot$ For $J_2 = \{2\}$ and $J_3= \{3\}$ the situation is
similar.

$\centerdot$ For $J_{1,2} = \{1, 2\}$, a system of generators of
\[Syz_{\mathbb{Q}[x_1]}[\sigma^{\gamma_1}(lc(\boldsymbol{f}_1))c_{\gamma_1, \beta_1} \hspace{0.3 cm} \sigma^{\gamma_2}(lc(\boldsymbol{f}_2))c_{\gamma_2, \beta_2}],\]
where $\beta_1 =$ $\exp(lm(\boldsymbol{f}_1))$, $\beta_2 =$
$\exp(lm(\boldsymbol{f}_2))$, $\gamma_1 =$
$\exp(\boldsymbol{X}_{J_{1, 2}}) - \beta_1$ and $\gamma_2 =$
$\exp(\boldsymbol{X}_{J_{1, 2}}) - \beta_2$, is $B^{J_{1, 2}} =
\{(4, - \frac{9}{4}x_1)\}$, thus we have only one generator
$\boldsymbol{b}_1^{J_{1, 2}} = (b_{11}^{J_{1, 2}}, b_{12}^{J_{1,
2}}) = (4, - \frac{9}{4}x_1)$ and
\begin{align*}
\boldsymbol{s}_1^{J_{1, 2}} &= b_{11}^{J_{1, 2}}x^{\gamma_1}\boldsymbol{\tilde{e}}_1 + b_{12}^{J_{1, 2}}x^{\gamma_2}\boldsymbol{\tilde{e}}_2 \\
&= 4x_3 \boldsymbol{\tilde{e}}_1 - \frac{9}{4}x_1x_2\boldsymbol{\tilde{e}}_2\\
&= \begin{pmatrix}
4x_3\\
- \frac{9}{4}x_1x_2\\
0
\end{pmatrix}.
\end{align*}
Then,
\[Syz(L_G) = \left \langle \begin{pmatrix}
4x_3\\
- \frac{9}{4}x_1x_2\\
0
\end{pmatrix} \right \rangle,\]
or in a matrix notation
\[Syz(L_G) = Z(L_G) = \begin{bmatrix}
4x_3\\
- \frac{9}{4}x_1x_2\\
0
\end{bmatrix}. \]
\item[(ii)] Next we compute $Syz(G)$:
By Division Algorithm we have
\[4x_3\boldsymbol{f}_1 - \frac{9}{4}x_1x_2\boldsymbol{f}_2 + 0\boldsymbol{f}_3 = p_{11}\boldsymbol{f}_1 + p_{21}\boldsymbol{f}_2 + p_{31}\boldsymbol{f}_3,\]
so by the Example \ref{231}, $p_{11} = 0 = p_{21}$ and $p_{31} =
1$, i.e., $P=\boldsymbol{\tilde{e}}_3$. Thus,
\begin{align*}
Z(G) &= Z(L_G) - P \\
&= \begin{bmatrix}
4x_3\\
- \frac{9}{4}x_1x_2\\
- 1
\end{bmatrix}
\end{align*}
and
\[Syz(G) = \left \langle \begin{pmatrix}
4x_3\\
- \frac{9}{4}x_1x_2\\
- 1
\end{pmatrix} \right \rangle. \]
\item[(iii)] Finally we compute $Syz(F)$:
since
\[\boldsymbol{f}_1 = 1\boldsymbol{f}_1 + 0\boldsymbol{f}_2 + 0\boldsymbol{f}_3, \ \boldsymbol{f}_2 = 0\boldsymbol{f}_1 + 1\boldsymbol{f}_2 + 0\boldsymbol{f}_3\]
then
\[Q = \begin{bmatrix}
1 &0\\
0 &1\\
0 &0
\end{bmatrix}.\]
Moreover,
\[\boldsymbol{f}_1 = 1\boldsymbol{f}_1 + 0\boldsymbol{f}_2, \ \boldsymbol{f}_2 = 0\boldsymbol{f}_1 + 1\boldsymbol{f}_2, \
\boldsymbol{f}_3 = 4x_3\boldsymbol{f}_1 -
\frac{9}{4}x_1x_2\boldsymbol{f}_2,\] hence
\[H = \begin{bmatrix}
1 &0 &4x_3\\
0 &1 &- \frac{9}{4}x_1x_2
\end{bmatrix}.\]

By Theorem \ref{syzygiesmainresult},
\[Syz(F) = \begin{bmatrix}
(Z(G)^TH^T)^T &I_2 - (Q^TH^T)^T
\end{bmatrix},\]
with
\begin{align*}
(Z(G)^TH^T)^T &= \left(\begin{bmatrix} 4x_3 &- \frac{9}{4}x_1x_2
&- 1
\end{bmatrix}
\begin{bmatrix}
1 &0\\
0 &1\\
4x_3 &- \frac{9}{4}x_1x_2
\end{bmatrix}\right)^T \\
&= \left(\begin{bmatrix} 0 &0
\end{bmatrix}\right)^T
= \begin{bmatrix}
0\\
0
\end{bmatrix}
\end{align*}
and
\[I_2 - (Q^TH^T)^T = \begin{bmatrix}
0 &0\\
0 &0
\end{bmatrix}.\]
From this we conclude that $Syz(F) = 0$. Observe that this means
that $M$ is free.
\end{enumerate}
\end{example}



\begin{flushright}
Departamento de Matemáticas\\
Universidad Nacional de Colombia\\
Bogotá, Colombia\\
\textit{e-mail}: \texttt{jolezamas@unal.edu.co}\\
\end{flushright}

\begin{thebibliography}{9}

\bibitem{Loustaunau}\textbf{Adams, W. and Loustaunau, P.},  \textit{An Introduction to Gröbner Bases},
Graduate Studies in Mathematics, AMS, 1994.

\bibitem{Bell}\textbf{Bell, A. and Goodearl, K.}, \textit{Uniform rank over differential operator rings and
Poincaré-Birkhoff-Witt extensons}, Pacific Journal of Mathematics,
131(1), 1988, 13-37.

\bibitem{Gomez-Torrecillas}\textbf{Bueso, J., Gómez-Torrecillas, J. and Lobillo, F.J.}, \textit{Homological computations in PBW modules},
Algebras and Representation Theory, 4, 2001, 201-218.

\bibitem{Gomez-Torrecillas2}\textbf{Bueso, J., Gómez-Torrecillas, J. and Verschoren, A.}, \textit{Algorithmic Methods in
Non-Commutative Algebra: Applications to Quantum Groups}, Kluwer,
2003.

\bibitem{Gallego2}\textbf{Gallego, C. and Lezama, O.}, \textit{Gröbner bases for ideals of $\sigma-PBW$ extensions},
Communications in Algebra, 39, 2011, 1-26.

\bibitem{Levandovskyy} \textbf{Levandovskyy, V.}, \textit{Non-commutatve Computer Algebra for Polynomial Algebras: Gröbner Bases,
Applications and Implementation}, Doctoral Thesis, Universität
Kaiserslautern, 2005.

\bibitem{Lezama2}\textbf{Lezama, O.}, \textit{Gröbner bases for modules over Noetherian polynomial commutative rings},
Georgian Mathematical Journal, 15, 2008, 121-137.

\bibitem{Lezama3}\textbf{Lezama, O. and Reyes, M.}, \textit{Some homological properties of skew PBW extensions}, Communications in Algebra, 42, 2014, 1200-1230.

\end{thebibliography}
\end{document}